\documentclass[12pt]{amsart}
\usepackage{amssymb,amsmath,amsthm,mathrsfs,multirow,xcolor,framed,url}
\usepackage[pdftex,
         pdfproducer={Latex with hyperref},
         pdfcreator={pdflatex}]{hyperref}
\newtheorem{theorem}{Theorem}[section]
\newtheorem{lemma}[theorem]{Lemma}
\newtheorem{cor}[theorem]{Corollary}
\newtheorem{conj}[theorem]{Conjecture}

\theoremstyle{definition}

\newtheorem*{remark}{Remark}

\numberwithin{equation}{section}

\newcommand{\abs}[1]{\lvert#1\rvert}
\newcommand{\spt}{\mbox{\rm spt}}

\newcommand{\Parans}[1]{\left(#1\right)}

\newcommand\leg[2]{\genfrac{(}{)}{}{}{#1}{#2}} 

\newcommand\FL[1]{\left\lfloor#1\right\rfloor}

\allowdisplaybreaks
\newcommand\mylabel[1]{\label{#1}}

\newcommand\myeqn[1]{(\ref{eq:#1})\framebox{#1}}

\newcommand{\beqs}{\begin{equation*}}
\newcommand{\eeqs}{\end{equation*}}
\newcommand{\beq}{\begin{equation}}
\newcommand{\eeq}{\end{equation}}
\renewcommand{\MR}[1]{\href{http://www.ams.org/mathscinet-getitem?mr={#1}}{MR{#1}}}

\newcommand\Up[2]{U_{#1}\Parans{#2}}
\newcommand{\Np}{N_\psi}
\newcommand{\Npm}{N_{\phi_{-}}}
\newcommand\nutwid{\overset {\text{\lower 3pt\hbox{$\sim$}}}\nu}
\newcommand\sgtwid{\overset {\text{\lower 3pt\hbox{$\sim$}}}{sg}}
\newcommand\Nptwid{\overset {\text{\lower 3pt\hbox{$\sim$}}}{\Np}}
\newcommand\eptwid{\overset {\text{\lower 3pt\hbox{$\sim$}}}{\varepsilon}}
\newcommand\myeqref[1]{\eqref{#1}}
\newcommand\myref[1]{\ref{#1}}
\newcommand{\sg}{\mbox{sg}}

\begin{document}

\title[Unimodal mod $4$ conjectures and mock theta functions]{A proof of 
the mod $4$ unimodal sequence conjectures 
and related mock theta functions}

\author{Rong Chen}
\address{School of Mathematical Sciences, East China Normal University,
Shanghai, People's Republic of China}
\email{rchen@stu.ecnu.edu.cn}
\author{F. G. Garvan}
\address{Department of Mathematics, University of Florida, Gainesville,
FL 32611-8105}
\email{fgarvan@ufl.edu}
\thanks{The first author was supported in part by the National Natural 
Science Foundation of China (Grant No. 11971173) and an ECNU 
Short-term Overseas Research Scholarship for Graduate Students
(Grant no. 201811280047).
The second author was supported in part by a grant from
        the Simon's Foundation (\#318714).}



\date{\today}                    

\begin{abstract}
In 2012 Bryson, Ono, Pitman and Rhoades showed how the generating functions
for certain strongly unimodal sequences are related to quantum modular
and mock modular forms. They proved some parity results and conjectured
some mod 4 congruences for the coefficients of these generating functions.
In 2016 Kim, Lim and Lovejoy obtained similar results for odd-balanced
unimodal sequences and made similar mod 4 conjectures. We prove
all of these mod 4 conjectures and similar congruences for the Andrews 
spt-function and related mock theta functions. Our method of proof involves 
new Hecke-Rogers type identities for indefinite binary quadratic forms 
and the Hurwitz class number.
\end{abstract}

\dedicatory{Dedicated to the memory of Freeman Dyson}

\maketitle

\section{Introduction}
\label{sec:intro}

We prove some conjectured congruences mod $4$ for certain unimodal
sequences. We follow Bryson, Ono, Pitman and Rhoades's \cite{Br-On-Pi-Rh12}
definition
of a strongly unimodal sequence.
A sequence of integers $\{a_j\}_{j=1}^s$ is a 
\textbf{strongly unimodal sequence} of size $n$ if it satisfies
$$
0 < a_1 < a_2 < \cdots < a_k > a_{k+1} > \cdots > a_s >0 \quad\mbox{and}\quad
a_1+a_2 + \cdots + a_s =n,
$$
for some $k$. Let $u(n)$ be the number of such sequences, and define the
\textbf{rank} of such a sequence as $s - 2k + 1$; i.e.\ the number terms
after the maximum minus the number of terms before it. Let $u(m,n)$ be
the number of strongly unimodal sequences of size $n$ and rank $m$. Then
\beq
\mathcal{U}(z;q) := \sum_{m,n} u(m,n) z^m q^n 
= \sum_{n=0}^\infty (-zq;q)_n (-z^{-1}q;q)_n q^{n+1},
\label{eq:Uzqdef}
\eeq
where we  use the usual $q$-notation
$$
(a;q)_n := \prod_{k=0}^{n-1}(1-aq^k).
$$
In addition, let $u(a,b;n)$ be the number of strongly unimodal sequences 
of size $n$ and rank congruent to $a$ mod $b$.

Bryson, Ono, Pitman and Rhoades \cite{Br-On-Pi-Rh12} relate
$\mathcal{U}(-1;q)$ with a quantum modular form which is dual to
a quantum modular form of Zagier \cite{Za2001}. They also show that
$\mathcal{U}(\pm i;q)$ is mock modular form and prove interesting parity results
for the coefficients $u(n)$. They made the following conjecture.
\begin{conj}
\label{conju}
Suppose $\ell\equiv 7,11,13,17\pmod{24}$ is prime and
 $\displaystyle\leg{k}{\ell}=-1$. Then for all $n$ we have
\beq
u(\ell^2n+kl-s(\ell))\equiv 0\pmod4,
\label{eq:conjuA}
\eeq
where $s(\ell)=\tfrac{1}{24}(\ell^2-1)$.
Moreover, for $a\in\{0,1,2,3\}$ we have
\begin{align}
u(a,4; \ell^2n+kl-s(\ell))&\equiv 0\pmod2,
\label{eq:conjuB}\\
\noalign{and}
u(0,4; \ell^2n+kl-s(\ell))&\equiv 
u(2,4; \ell^2n+kl-s(\ell))\pmod4.
\label{eq:conjuC}
\end{align}
\end{conj}
\begin{remark}
Here and throughout this paper $\displaystyle\leg{\cdot}{\cdot}$
denotes the Kronecker symbol.
\end{remark}

In this paper we prove Conjecture \myref{conju} and a similar conjecture
for a related unimodal sequence function studied by 
Kim, Lim and Lovejoy \cite{Ki-Li-Lo16}. We note that there is a stronger
result
for \myeqref{eq:conjuC} which include primes congruent to 
$1,\pm5\pmod{24}$. See Theorem \myref{conjuCstrong}.
We also prove analogous mod $4$ results
for Andrews's \cite{An08} spt-function and for the coefficients of related
mock theta functions.

We describe Kim, Lim and Lovejoy's odd-balanced unimodal sequence.
A sequence of integers $\{a_j\}_{j=1}^s$ is \textbf{unimodal}
 of size $n$ if it satisfies
$$
0 < a_1 \le a_2 \le \cdots \le a_{k-1} < a_k > a_{k+1} \ge \cdots 
\ge a_{s-1} \ge a_s >0 \quad\mbox{and}\quad
a_1+a_2 + \cdots + a_s =n,
$$
Such a unimodal sequence is called \textbf{odd-balanced}
if the peak $a_k$ is even, even parts to the left
and right of the peak are distinct and the odd parts to the left of the peak
are identical with those to the right. As before the rank is the number to 
right of the peak minus the number to the left. We let
$v(n)$ be the number of odd-balanced unimodal sequences of size $2n+2$
and let $v(m,n)$ be the number with rank $m$.
Then
$$
\mathcal{V}(z;q) := \sum_{m,n} v(m,n) z^m q^n 
= \sum_{n=0}^\infty \frac{(-zq;q)_n (-z^{-1}q;q)_n q^{n}}{(q;q^2)_{n+1}}.
$$
Kim, Lim and Lovejoy obtained quantum modular, mock modular and parity
results analogous to Bryson, Ono, Pitman and Rhoades's results. 
They point out the mock modularity properties of $\mathcal{V}(z;q)$
follow from some results of Mortenson \cite[Theorem 4.4]{Mo2014}.
They also made the
following 
\begin{conj}
\label{conjv}
Let $\ell\not\equiv -1\pmod8$ be prime and let $k$ be a positive integer 
satisfying $\ell\|8k+7$. If $(8k+7)/\ell$ is a quadratic residue modulo $\ell$, then
$$
v(\ell^2n+k)\equiv 0\pmod4.
$$
For example, $v(9n+4)\equiv v(25n+11)\equiv v(25n+16)\equiv 0\pmod4$. 
Moreover, the same congruences hold for the coefficients of 
$\mathcal{V}(\pm i,q)$.
\end{conj}
\begin{remark}
This conjecture is false. 
For example let $\ell=17$ and $k=99$. Then $(8k+7)/\ell=47$ is a quadratic
residue mod $17$, but
$$
v(99)=81474897186\equiv 2\pmod4.
$$
The primes congruent  to $1$ mod $8$ must be excluded. The correct version
of this conjecture, which we prove, is given below in Theorem \myref{thev4}.
\end{remark}

The Andrews's spt-function \cite{An08} is related to strongly unimodal 
sequences.  
See equations \myeqref{eq:Upsiid} and \myeqref{eq:sptPsimod4} below.
Andrews defined $\spt(n)$ as the number of smallest parts in the partitions 
of $n$,
and  proved that it satisfied some surprising Ramanujan-type congruences 
mod $5$, $7$ and $13$. Bringmann \cite{Br2008} showed the spt generating function
is related to a weight $3/2$ harmonic Maass form.
The second author made the following
\begin{conj}\footnote{This conjecture was presented by the second author
in a talk, entitled \textit{The Andrews spt-function mod 4}, at the
AMS Special Session on Arithmetic Properties of Sequences from Number Theory and Combinatorics, AMS Annual Meeting, Atlanta, January 4, 2017.}
\label{sptconj}
Suppose $\ell > 3$ is prime and $\ell\not\equiv23\pmod{24}$.
Let $\eptwid=\eptwid(\ell) =  1$ if $\ell\equiv 1\pmod{24}$ and $-1$ otherwise. 
Then
$$
\spt(\ell n - s({\ell})) \equiv 0 \pmod{4}, \quad
\mbox{(where $s(\ell)=\tfrac{1}{24}(\ell^2-1)$)},
$$
when $\displaystyle\leg{n}{\ell} = \eptwid$.
\end{conj}
We prove this conjecture in Section \myref{subsec:spt}.

The three conjectures above are related to some mod $4$ properties for 
the coefficients of certain mock theta functions. 
Mock theta functions were first defined by Ramanujan in his famous 
last letter to Hardy \cite[pp.354-355]{Ra2000}.
In this letter he gave examples of 17 such functions:
four of order 3, ten of order 5 and three of order 7.
Since then more 
mock theta functions have been found by others. See for example Watson
\cite{Wa1936}, Andrews and Hickerson \cite{An-Hi91}, Gordon and
McIntosh \cite{Go-McI00}, Berndt and Chan \cite{Be-Ch07} and McIntosh
\cite{McI07}. Many Ramanujan-type
congruences for mock theta functions have been found mainly to 
modulus relatively prime to $6$. See for example \cite{An-Pa-Se-Ye2017}.
In a quite recent paper
Wang \cite{Wa2020} obtained some extensive  parity results for 
$21$ of the $44$ classical
mock theta functions. In this paper we consider mock theta congruences
mod $4$. The mod $4$ congruences are much harder to prove.
Conjecture \myref{conju} is related to Ramaujan's third
order mock theta functions $\psi(q)$. In fact,
\beq
\mathcal{U}(\pm i;q) = \psi(q) = 
\sum_{n=1}^\infty N_\psi(n) q^n =
\sum_{n=1}^\infty \frac{q^{n^2}}{(q;q^2)_n}.
\label{eq:Upsiid}
\eeq
The function $\psi(q)$ is also related to the spt-function. Andrews,
Liang and the second author \cite{An-Ga-Li13} showed that
\beq
\spt(n) \equiv (-1)^{n-1} N_\psi(n) \pmod{4}.
\label{eq:sptPsimod4}
\eeq
This means that $N_\psi(n)$ also satisfies Conjecture \myref{sptconj}.

Conjecture \myref{conjv} is related to the second order mock theta function
\beq
\label{Aqdef}
A(q)=\sum_{n=0}^\infty N_A(n) q^n = 
\sum_{n=0}^{\infty}\frac{q^{(n+1)^2}(-q;q^2)_n}{(q;q^2)_{n+1}^2}
=\sum_{n=0}^{\infty}\frac{q^{n+1}(-q^2;q^2)_n}{(q;q^2)_{n+1}}.
\eeq
In fact,
$$
q \mathcal{V}(\pm i;q) = A(q),
$$
and
Kim, Lim and Lovejoy \cite[Conjecture 1.3]{Ki-Li-Lo16} made the following
conjecture.
\begin{conj}
\label{conjAq}
Let $p\not\equiv7\pmod8$ be an odd prime, suppose $8\delta_p\equiv1\pmod{p^2}$ and $k,n\in \mathbb{Z}$ where $\displaystyle\leg{k}{p}=1$. Then
\beq
\label{type4}
N_A(p^2n+(pk+1)\delta_p)\equiv 0\pmod4.
\eeq
\end{conj}

In a previous paper \cite{Ch-Ga} we found some weight $3/2$ eta-products 
that satisfy a similar mod $4$ behaviour. For example, 
let
$$
f(q):=\sum_{n=0}^\infty a(n) q^n:=\frac{J_3^3J_2^2}{J_1^2},
$$
where we have used the common notation
$$
J_k:=(q^k;q^k)_\infty:=\prod_{n=1}^\infty (1-q^{kn}),\quad\mbox{and}\qquad
(z;q)_\infty := \prod_{n=1}^\infty (1-z q^{n-1}).
$$
The following is a typical result we found.
\begin{theorem}[{\cite[Theorem 1.3]{Ch-Ga}}]
Let $p>3$ be prime. Suppose $24\delta_p\equiv1\pmod{p^2}$, and $k$,
$n\in\mathbb{Z}$ where $\displaystyle\leg{k}{p} = 1$.
Then
\begin{align*}
a(p^2n+(pk-11)\delta_p)&\equiv 0 \pmod4,\qquad\mbox{if $p\not\equiv
11\pmod{24}$}.
\end{align*}
\end{theorem}
A crucial aspect of the proof involves the connection with
ternary quadratic forms and the class number of imaginary quadratic
fields. The approach in this paper is similar and involves the
Hurwitz class number.



In Section \myref{sec:class} we collect and prove some needed results
for the Hurwitz class number function $H(N)$. In Section \myref{sec:mock}
we study four mock theta functions related to the Hurwitz
class number function: $A(q)$ (second order), $\phi_{-}(q)$ and 
$\sigma(q)$ (sixth order), 
and $V_1(q)$ (eighth order), and prove some surprising congruences mod $4$ for
their coefficients. In Section \myref{sec:seq} we study the
third order mock theta function $\psi(q)$, and we prove the main conjectures
for strongly unimodal sequences, odd-balanced unimodal sequences and
the Andrews spt-function.

In this paper, we will often use the following elementary
congruences.                                  
\beq
\left\{\begin{matrix}
\displaystyle\frac{J_1^2}{J_2}=1+2\sum_{n=1}^{\infty}(-1)^nq^{n^2}\equiv 1\pmod2, 
&\quad
\displaystyle\frac{J_2^5}{J_4^2J_1^2}=1+2\sum_{n=1}^{\infty}q^{n^2}\equiv 1\pmod2,
\\
\mbox{and}\quad \displaystyle\frac{J_1^4}{J_2^2}\equiv 1\pmod4. &
\end{matrix}
\right.
\label{eq:thetamod4}
\eeq
These follow
from the well-known Jacobi triple product identity 
\cite[Theorem 3.4, p.461]{An1974}:
\beq
(z;q)_\infty (z^{-1}q;q)_\infty (q;q)_\infty
= \sum_{n=-\infty}^\infty (-1)^n z^n q^{n(n-1)/2}.
\label{jacp}
\eeq
We will also use the following 
special cases of \myeqref{jacp}.
\begin{align}
\sum_{n=0}^{\infty}q^{n(n+1)/2}&=\frac{J_2^2}{J_1},
\label{eq:triangprod}\\
\sum_{k=-\infty}^\infty q^{k(3k+1)/2}&=\frac{J_3^2 J_2}{J_6 J_1},
\label{eq:pentprod}
\end{align}
As well as the triple product identity we will use the quintuple 
product identity \cite[Theorem 3.9, p.467]{An1974}:
\begin{align}
& (-z;q)_\infty (-z^{-1}q;q)_\infty (qz^{2};q^2)_\infty (z^{-2}q;q^2)_\infty
\mylabel{eq:quintprod}\\
& \qquad =\sum_{n=-\infty}^\infty (-1)^n z^{3n} q^{n(3n-1)/2} +
\sum_{n=-\infty}^\infty (-1)^n z^{3n+1} q^{n(3n+1)/2}.
\nonumber                   
\end{align}

\section{The Hurwitz class number}
\label{sec:class}

Following 
\cite[Section 5.3]{Co93} we define the Hurwitz class number $H(N)$, where
$N$ is a non-negative integer, as follows.
\begin{enumerate}
\item[(1)] If $N\equiv 1,2\pmod4$ then $H(N)=0$.
\item[(2)] If $N=0$ then $H(0)=-1/12$.
\item[(3)] If $N>0$, $N\equiv 0,3\pmod4$, then $H(N)$ is the class number of positive definite binary quadratic forms of discriminant $-N$, with those classes that contain a multiple of $x^2 + y^2$ or $x^2 + xy + y^2$ counted with weight $1/2$ or $1/3$, respectively.
\end{enumerate}
It is known that (see for example \cite{Wi18})
\beq
\label{Hn}
H(n)=\frac{2h(D)}{\omega(D)}\sum_{d\mid f}\mu(d)\leg{D}{d}\sigma_1(f/d),
\eeq
if $n=-Df^2$, where $-D$ is a fundamental discriminant,  $h(D)$ is the class
 number of $\mathbb{Q}(\sqrt{D})$, 
$\mu$ is the M\"obius function,    
$\sigma_1$ is the divisor sum, and
$\omega(D)$ is the number of units in the ring of integers of 
$\mathbb{Q}(\sqrt{D})$. As mentioned before $\displaystyle\leg{\cdot}{\cdot}$ is
 the Kronecker symbol.
In particular for $f=1$ we have              
\beq
\label{HD}
H(-D)=\frac{2h(D)}{\omega(D)}.
\eeq

\subsection{Congruences}
\label{subsec:congs}

The main goal in this subsection is to characterize congruences mod $2$ and $4$ 
for $H(4n+3)$. The method of proof is completely analogous to the methods
of \cite{Ch-Ga}, where the same type of results were found for 
certain weight $3/2$ eta-products. We provide analogues of the results
in \cite{Ch-Ga} needed to prove our characterizations of the mod $2$ and
mod $4$ congruences for $H(4n+3)$. These characterizations are given in
Theorem \myref{thm:the12}.

We recall the following theorems.

\begin{theorem}
\label{thm:Hngenus}
We have $3H(3)=1$ and if $n$ is square-free, $n>3$ and $n\equiv 3\pmod4$, 
then we have
$$
H(n)=h(-n)=2^{t-1}k,
$$
where $t$ is the number of distinct prime factors of $n$ and $k$ is the
number of classes in each genus of $\mathbb{Q}(\sqrt{-n})$.
\end{theorem}
\begin{remark}
This theorem follows immediately from Proposition 3.11, Corollary 3.14
and Theorem 3.15 in \cite{Cox1989}.
\end{remark}

\begin{theorem}
\label{thm:HnKRONsum}
If $n$ is square-free and $n\equiv 3\pmod4$, then we have
\beq
\label{3Hn}
\Parans{2-\leg{n}{2}}H(n)=\Parans{2-\leg{n}{2}}h(-n)
=\sum_{r=1}^{(n-1)/2}\leg{r}{n}.
\eeq
\end{theorem}
\begin{remark}
This theorem follows from \cite[Theorem 3, p.346]{Bo-Sh66}. We note that
$$
\leg{n}{2} = 
\begin{cases}
-1 & \mbox{if $n\equiv3\pmod{8}$},\\
1 & \mbox{if $n\equiv7\pmod{8}$}.  
\end{cases}
$$
\end{remark}

The following is an  analog of \cite[Theorem 3.6]{Ch-Ga}.
\begin{theorem}
\label{thm:Hnleg}
Suppose $n\equiv 3\pmod4$ is square-free. We have
\begin{enumerate}
\item[(1)] $3H(n)$ is odd if and only if $n$ is a prime,
\item[(2)] $3H(n)\equiv 2 \pmod4$ if and only if $n=p_1p_2$ is a product of
two primes which satisfy
$$
\leg{p_1}{p_2}=-1.
$$
\end{enumerate}
\end{theorem}
\begin{remark}
The result (1) is well-known and follows from Theorem \myref{thm:Hngenus}.  
The result (2) follows from 
Theorem \myref{thm:Hngenus} together with Propositions 3 and 4 in \cite{Pizer1976}.
It may also be proved using the method in \cite[Section 3]{Ch-Ga}.
As pointed out by Pizer \cite[p.189]{Pizer1976}, (2) is also due 
to Hasse \cite{Hasse1970}.
\end{remark}

The Hurwitz class number function  also satisfies 
the analog of \cite[eq.(4.1)]{Ch-Ga}. 

\begin{lemma}
\label{theHid}
For each odd prime $p$, we have
\beq
\label{eigen}
H(p^2n)+\leg{-n}{p} H(n)+pH(n/p^2)=(p+1)H(n),
\eeq
for all $n\ge0$.
\end{lemma}
\begin{remark}
It is understood that $H(n)=0$ if $n$ is non-integral or negative.
This result is known. It is noted in the proof
of Proposition 5.1 in \cite{Ah-Br-Lo11}. 
Ahlgren, Bringmann and Lovejoy \cite{Ah-Br-Lo11} prove
that the generating function for $H(n)$ is a Hecke eigenform by using the
fact that the generating function is mock modular form of weight $3/2$.
For completeness we provide an elementary proof that only uses
\myeqref{Hn}.                                            
This argument was also observed by
Beckwith, Raum and  Richter \cite{Be-Ra-Ri2020}.
\end{remark}
\begin{proof}
Suppose $p$ is an odd prime. We may assume $n\equiv0$ or $3\pmod{4}$.
Then, by \myeqref{Hn} and \myeqref{HD}, we have
$$
H(n)=H(-D)G(D,f)
$$
where $n=-Df^2$, $D$ is a fundamental discriminant of $\mathbb{Q}(\sqrt{-n})$, namely $-D$ is square-free or $-D=4m$, $m\equiv 1,2\pmod4$ and $m$ is square-free and
$$
G(D,f):=\sum_{d\mid f}\mu(d)\leg{D}{d}\sigma_1(f/d).
$$
It is clear that $G(D,f$) is a multiplicative function of  $f$.
We consider two cases.
\subsubsection*{Case 1} $(p,f)=1$. Then $n/p^2\not\in\mathbb{Z}$ and $H(n/p^2)=0$.
\begin{align*}
H(p^2n)&=H(-D)G(D,pf)=H(-D)G(D,p)G(D,f)\\
&=H(-D)\left(\sigma_1(p)-\leg{D}{p}\right)G(D,f)=H(-D)\left(1+p-\leg{D}{p}\right)G(D,f)\\
&=-\leg{D}{p}H(n)+(1+p)H(n),
\end{align*}
so that \myeqref{eigen} holds. 

\subsubsection*{Case 2} $(p,f)>1$. Then we let $f=p^\alpha f_1$ where $\alpha \geq 1$ 
and $(p,f_1)=1$. We have
\begin{align*}
&H(p^2n)=H(-D)G(D,pf)=H(-D)G(D,p^{\alpha+1})G(D,f_1),\\
&H(n)=H(-D)G(D,f)=H(-D)G(D,p^{\alpha})G(D,f_1),\\
&H(n/p^2)=H(-D)G(D,f/p)=H(-D)G(D,p^{\alpha-1})G(D,f_1).
\end{align*}
So we need to show
$$  
G(D,p^{\alpha+1})+pG(D,p^{\alpha-1})-(1+p)G(D,p^{\alpha})=0
$$        
since $p\mid n$ and $\displaystyle\leg{-n}{p}=0$. 
This is an easy exercise since 
$$                   
G(D,p^\beta) = \sigma_1(p^\beta) - \leg{D}{p} \sigma_1(p^\beta),\quad\mbox{and}\quad
\sigma_1(p^{\beta+2}) = p^{\beta+2} + p^{\beta+1} + \sigma_1(p^\beta),
$$                 
for all $\beta>0$. We have \myeqref{eigen} in this case.
\end{proof}

By Theorem \myref{thm:Hnleg} and  Lemma \myref{theHid}, we have the following theorem 
by a  proof that is analogous to that of 
Theorem 4.3 and Theorem 4.6 in \cite{Ch-Ga}.

\begin{theorem}
\label{thm:the12}
For $n\equiv3 \pmod4$,
\begin{enumerate}
\item[(1)] $3H(n)$ is odd if and only if $n$ has the form
$$
n=p^{4a+1}m^2,
$$
where $p$ is prime, and $m$ and $a$ are integers satisfying
$(m,p)=1$ and $a\ge0$.
\item[(2)]  $3H(n)\equiv 2\pmod4$ if and only if $n$ has the form
$$
n=p_1^{4a+1}p_2^{4b+1}m^2,
$$
where $p_1$ and $p_2$ are primes such that $\displaystyle\leg{p_1}{p_2}=-1$,
$(m,p_1p_2)=1$ and $a,b\geq 0$ are integers.
\end{enumerate}
\end{theorem}
\begin{remark}
The factor $3$ 
that appears in Theorem \myref{thm:Hnleg} and Theorem \myref{thm:the12} is 
only to guarantee that $3H(n)$ is an integer.
\end{remark}

\subsection{Generating functions}
\label{subsec:genfuncs}

In this subsection we give identities for certain generating functions
$$
\mathscr{H}_{a,b}(q):=\sum_{n=0}^{\infty}H(an+b)q^n,
$$
where $a$ is a divisor of $24$ and $(a,b)=1$.

Gauss's Three Squares Theorem (see \cite[Theorem 1.5]{Mo17}) shows that
$$
24\mathscr{H}_{8,3}(q)=\sum_{n=0}^{\infty}r_3(8n+3)q^n,
$$
where $r_3(n)$ denotes the number of representations of $n$ as a sum of $3$ squares. Noting that $x^2+y^2+z^2\equiv 3\pmod8$ if and only if $x$, $y$, $z$ are odd, 
we have 
\begin{align*}
\sum_{n=0}^{\infty}r_3(8n+3)q^{8n+3}&=\sum_{x,y,z\in\mathbb{Z}}q^{(2x+1)^2+(2y+1)^2+(2z+1)^2}
=8\sum_{x,y,z\geq0}q^{(2x+1)^2+(2y+1)^2+(2z+1)^2}\\
&=8q^3\sum_{x,y,z\geq0}q^{4(x^2+x+y^2+y+z^2+z)}
=8q^3\left(\sum_{n=0}^{\infty}q^{4n(n+1)}\right)^3.
\end{align*}
See also \cite[Eq.(1.14), p.41]{Wa1935}. 
Hence
\beq
\label{H83}
3\mathscr{H}_{8,3}(q)=\left(\sum_{n=0}^{\infty}q^{n(n+1)/2}\right)^3
=\Parans{\frac{J_2^2}{J_1}}^3 = \frac{J_2^6}{J_1^3}
\eeq
by \myeqref{eq:triangprod}.
See also \cite[Eq.(1.14), p.41]{Wa1935}. 
By calculating the $3$-dissection of $\mathscr{H}_{8,3}(q)$ and using 
\myeqref{eq:pentprod} 
we find that         
\beq
\label{p11}
\mathscr{H}_{24,11}(q)=\left(\sum_{n=-\infty}^{\infty}q^{n(3n+1)/2}\right)^2\sum_{n=0}^{\infty}q^{3n(n+1)/2}=\frac{J_3^3J_2^2}{J_1^2},
\eeq
and
\beq
\label{p19}
\mathscr{H}_{24,19}(q)=\sum_{n=-\infty}^{\infty}q^{n(3n+1)/2}\left(\sum_{n=0}^{\infty}q^{3n(n+1)/2}\right)^2=\frac{J_6^3J_2}{J_1}.
\eeq
We let $t(n)$ denote  the number of representations of $n$
by the ternary quadratic form $x^2+3y^2+3z^2$. Then Bringmann and Kane
\cite[p.3]{Br-Ka2020} found that
\beq
t(n)=8\left(1+\left(\tfrac{n}{3}\right)\right)H(n),
\label{eq:p7a}
\eeq
when $n\equiv 7\pmod{8}$ with $9\nmid n$. Equation \myeqref{eq:p7a} can also 
be
proved using  \cite[Lemma 4.14]{Br-Ka2020} and Lemma \myref{theHid}. From
equation \myeqref{eq:p7a} it can be shown that
\beq
\label{p7}
\mathscr{H}_{24,7}(q)=\sum_{n=-\infty}^{\infty}q^{n(3n+1)/2}\left(\sum_{n=0}^{\infty}q^{n(n+1)/2}\right)^2=\frac{J_3^2J_2^5}{J_6J_1^3}.
\eeq
We omit the details.
In \cite{Ch-Ga} we obtained congruences modulo $4$ for the coefficients
of the eta-products in 
\myeqref{p11}, \myeqref{p19} and \myeqref{p7} using special cases of 
Theorem \myref{thm:the12}. 

We will need the following result of Humbert \cite[p.368]{Hu07}.
\beq
\label{H87}
\mathscr{H}_{8,7}(q)
=\frac{1}{qJ_1^3}\sum_{n=0}^{\infty}\frac{(-1)^{n+1}n^2q^{n(n+1)/2}}{1+q^n}.
\eeq
See also \cite[p.51]{Wa1935}. Naturally we find that $\mathscr{H}_{8,7}(q)$
is not an eta-product.
Humbert \cite[p.442]{Hu07} also found 
\beq
\label{H2423}
\mathscr{H}_{24,23}(q)=q^{-23/24}\sum_{(x,y,z)\in D}q^{(2x^2-3z^2-6y^2)/24},
\eeq
where
$$
D=\{(x,y,z)\in \mathbb{Z}^3:0<3z<2x,-x<3y<x,2x^2-3z^2-6y^2\equiv23\pmod{24}\}.
$$

\section{Congruences for a class of mock theta functions}
\label{sec:mock}

In this section we will discuss congruences for four mock theta functions 
associated with the Hurwitz class number functions
$\mathscr{H}_{4,3}$, $\mathscr{H}_{8,7}$, $\mathscr{H}_{12,11}$ 
and $\mathscr{H}_{24,23}$. These mock theta functions are 
$A(q)$ (second order), $\phi_{-}(q)$ and $\sigma(q)$ (sixth order), and
$V_1(q)$ (eighth order). The parity of the coefficients of these functions
and many other mock theta functions was recently found by Wang \cite{Wa2020}.
We determine their behaviour mod $4$ by relating them with the Hurwitz class
number.

\subsection{The second order mock theta function $A(q)$}

The second order mock theta function $A(q)$ is defined in \myeqref{Aqdef}. 
We find a congruence relation between the coefficients of $A(q)$ and the 
Hurwitz class number $H(8n-1)$. 
We have the following lemma.                        
\begin{lemma}
\label{lem:NAH}
For each integer $n>0$, we have
\beq
\label{eq:NAH}
N_A(n)\equiv (-1)^{n+1}H(8n-1)\pmod4.
\eeq
\end{lemma}

\begin{proof}
Ramanujan \cite[p.8]{Ra1988} found that $A(q)$ can be written as an 
Appell-Lerch sum
\beq
\label{Aqap}
A(q)=q\frac{(-q;q^2)_\infty}{(q^2;q^2)_\infty}\sum_{n=0}^{\infty}\frac{(-1)^nq^{2n^2+3n}}{1-q^{2n+1}}.
\eeq
For a proof see \cite[p.265]{An-Be2005}. 
Combining \myeqref{H87} and \myeqref{Aqap} and noting that $(2n)^2\equiv 0\pmod4$ and $(2n+1)^2\equiv 1\pmod4$, we have
\begin{align}
\label{AqH}
\mathscr{H}_{8,7}(q)&=\frac{1}{qJ_1^3}\sum_{n=0}^{\infty}\frac{(-1)^{n+1}n^2q^{n(n+1)/2}}{1+q^n}\equiv\frac{J_1^4}{J_2^2}\cdot\frac{1}{qJ_1^3}\sum_{n=0}^{\infty}\frac{(2n+1)^2q^{(2n+1)(n+1)}}{1+q^{2n+1}}\\
\nonumber
&\equiv\frac{J_1}{J_2^2}\sum_{n=0}^{\infty}\frac{q^{2n^2+3n}}{1+q^{2n+1}}=\frac{A(-q)}{-q}\pmod4,
\end{align}
by \myeqref{eq:thetamod4}.
This implies
$$
(-1)^{n+1}N_A(n)\equiv H(8n-1)\pmod4,
$$
which is \myeqref{eq:NAH}.
\end{proof}

We note that by \myeqref{eq:NAH}, $N_A(n)$ is odd if and only if $H(8n-1)$ is 
odd and $N_A(n)\equiv 2\pmod4$ if and only if $H(8n-1)\equiv 2\pmod4$. 
Also $m^2\equiv 1\pmod8$ for each odd $m$. Thus Theorem \myref{thm:the12} implies

\begin{theorem}
\label{the12A}
Let $n$ be a positive integer.
\begin{enumerate}
\item[(1)] $N_A(n)$ is odd if and only if $8n-1$ has the form
$$
8n-1=p^{4a+1}m^2,
$$
where $p$ is prime, and $m$ and $a$ are integers satisfying
$(m,p)=1$ and $a\ge0$.
\item[(2)]  $N_A(n)\equiv 2\pmod4$ if and only if $8n-1$ has the form
$$
8n-1=p_1^{4a+1}p_2^{4b+1}m^2,
$$
where $p_1$ and $p_2$ are primes such that $\displaystyle\leg{p_1}{p_2}=-1$,
$(m,p_1p_2)=1$ and $a,b\geq 0$ are integers.
\end{enumerate}
\end{theorem}
\begin{remark}
Wang \cite[Theorem 3.1]{Wa2020} proved (1) by a different method.
\end{remark}


\begin{theorem}
\label{theA4}
Kim, Lim and Lovejoy's Conjecture \myref{conjAq} is true.
\end{theorem}

\begin{proof}
Let $p\not\equiv7\pmod8$ be an odd prime
and suppose $n$, $k$ are nonnegative integers where $\displaystyle\leg{k}{p}=1$. Suppose
$m=p^2n+(pk+1)\delta_p$, where $\delta_p=\tfrac{1}{8}(7p^2+1)$.
Since
\beq
\label{8m-1}
8m-1\equiv pk\pmod{p^2},
\eeq
we have $p\|8m-1$.  Theorem \myref{the12A}(1)  implies that $N_A(m)$ 
is  even since $p\not\equiv7\pmod{8}$. Now suppose $N_{A}(m)\equiv 2\pmod4$. 
Then Theorem \myref{the12A}(2) and \myeqref{8m-1} imply that 
$$
8m-1=p^1 q^{4b+1}t^2,\quad\mbox{and}\quad
k\equiv q^{4b+1}t^2\pmod{p},
$$
where $q$ is a prime satisfying 
$\displaystyle\leg{p}{q}=-1$, 
$(pq,t)=1$, and $b\ge0$, $t>0$ are integers. 
Since $pq\equiv-1\pmod{8}$
either $p$ or $q\equiv1\pmod{4}$
so by quadratic reciprocity $\leg{q}{p}=\leg{p}{q}=-1$.
But
$$
\leg{q}{p}=\leg{q^{4b+1}t^2}{p}=\leg{k}{p}=1,
$$
which is a contradiction.
Hence
$N_A(m)\not\equiv 2\pmod4$, and   
\begin{equation*}
N_A(m)\equiv 0\pmod4. \qedhere
\end{equation*}
\end{proof}

\subsection{The eighth order mock theta function $V_1(q)$}

McIntosh \cite{McI07} studied the 
eighth order mock theta function 
$$
V_1(q)=
\sum_{n=0}^\infty N_{V_1}(n) q^n = 
\sum_{n=0}^{\infty}\frac{q^{(n+1)^2}(-q;q^2)_n}{(q;q^2)_{n+1}}.
$$
By \cite[p.286 Eq(4),Eq(7)]{McI07} we have
\beq
\label{V1q}
V_1(q)=A(q^2)+qP(q^2),
\eeq
where
\beq
\label{Pq1}
P(q)=(-q;q)_\infty(-q^2;q^2)_\infty^2(q^4;q^4)_\infty=\frac{J_4^3}{J_1J_2}.
\eeq
We have the following lemma for $N_{V_1}(n)$.
\begin{lemma}
For each integer $n>0$, we have
\beq
\label{NV1H}
N_{V_1}(n)\equiv \chi(n)H(4n-1)\pmod4,
\eeq
where $\chi(n)=-1$ if $n\equiv 0,1\pmod4$ and $\chi(n)=1$ if $n\equiv 2,3\pmod4$.
\end{lemma}

\begin{proof}
Equations \myeqref{eq:NAH} and \myeqref{V1q} give
$$
N_{V_1}(2n)=N_A(n)\equiv (-1)^{n+1}H(8n-1)\pmod4,
$$
which implies that \myeqref{NV1H} holds for even $n$. By \myeqref{eq:thetamod4}, 
\myeqref{H83} and \myeqref{Pq1} we have
$$
P(-q)=(q;-q)_\infty(-q^2;q^2)_\infty^2(q^4;q^4)_\infty
=\frac{J_4^4J_1^4}{J_2^{10}}\cdot\frac{J_2^6}{J_1^3}\equiv \frac{J_2^6}{J_1^3}
=3\mathscr{H}_{8,3}(q)\pmod4.
$$
Hence
$$
N_{V_1}(2n+1)\equiv (-1)^{n+1}H(8n+3)\pmod4,
$$
which implies that \myeqref{NV1H} holds for odd $n$.
\end{proof}

By \myeqref{NV1H}, $N_{V_1}(n)$ is odd if and only if $H(4n-1)$ is odd and $N_{V_1}(n)\equiv 2\pmod4$ if and only if $H(4n-1)\equiv 2\pmod4$. 
From Theorem \myref{thm:the12} we obtain

\begin{theorem}
\label{the12V1}
For any positive integer $n$,
\begin{enumerate}
\item[(1)] $N_{V_1}(n)$ is odd if and only if $4n-1$ has the form
$$
4n-1=p^{4a+1}m^2,
$$
where $p$ is prime, and $m$ and $a$ are integers satisfying
$(m,p)=1$ and $a\ge0$.
\item[(2)]  $N_{V_1}(n)\equiv 2\pmod4$ if and only if $4n-1$ has the form
$$
4n-1=p_1^{4a+1}p_2^{4b+1}m^2,
$$
where $p_1$ and $p_2$ are primes such that $\displaystyle\leg{p_1}{p_2}=-1$,
$(m,p_1p_2)=1$ and $a,b\geq 0$ are integers.
\end{enumerate}
\begin{remark}
The result (1) was also found by Wang \cite[Theorem 8.4]{Wa2020}.
\end{remark}
\end{theorem}


\begin{theorem}
\label{thmV1mod4}
Let $p\not\equiv3\pmod4$ be an odd prime, suppose $4\delta_p\equiv1\pmod{p^2}$ and $k,n\in \mathbb{Z}$ where $\displaystyle\leg{k}{p}=1$. Then
$$
N_{V_1}(p^2n+(pk+1)\delta_p)\equiv 0\pmod4.
$$
\end{theorem}

\begin{proof}
Let $p\not\equiv3\pmod4$ be an odd prime
and suppose $n$, $k$ are nonnegative integers where $\displaystyle\leg{k}{p}=1$. 
Suppose
$m=p^2n+(pk+1)\delta_p$, where $\delta_p=\tfrac{1}{4}(3p^2+1)$.
Since
\beq
\label{4m-1}
4m-1\equiv pk\pmod{p^2},
\eeq
we have $p\|4m-1$. Theorem \myref{the12V1}(1) implies that $N_{V_1}(m)$ is even
since $p\not\equiv3\pmod{4}$.
Now suppose $N_{V_1}(m)\equiv 2\pmod4$. Then Theorem \myref{the12V1}(2) and 
\myeqref{4m-1} imply that 
$$
4m-1=p^1 q^{4b+1}t^2,\quad\mbox{and}\quad
k\equiv q^{4b+1}t^2\pmod{p},
$$
where $q$ is a prime satisfying 
$\displaystyle\leg{p}{q}=-1$, 
$(pq,t)=1$, and $b\ge0$, $t>0$ are integers. 
Since $pq\equiv-1\pmod{4}$
either $p$ or $q\equiv1\pmod{4}$
so by quadratic reciprocity $\leg{q}{p}=\leg{p}{q}=-1$.
But
$$
\leg{q}{p}=\leg{q^{4b+1}t^2}{p}=\leg{k}{p}=1,
$$
which is a contradiction.
Hence $N_{V_1}(m)\not\equiv 2\pmod4$, and       
$$
N_{V_1}(m)\equiv 0\pmod4.
$$
\end{proof}

\subsection{The sixth order mock theta function $\phi_-(q)$}
\label{subsec:phiminus}

Berndt and Chan \cite{Be-Ch07} defined two new sixth order mock theta functions
including
$$
\phi_-(q)=\sum_{n=1}^\infty N_{\phi_{-}}(n) q^n :=
\sum_{n=1}^{\infty}\frac{q^n(-q;q)_{2n-1}}{(q,q^2)_n}.
$$
We will show that $\phi_-(q)$ is associated with $\mathscr{H}_{24,23}$. 
This is more difficult than showing $A(q)$ is associated with 
$\mathscr{H}_{8,7}$. 
Unfortunately Humbert's formula \myeqref{H2423} did not reveal  
the relations needed. Instead 
we find a relation mod $4$  between 
$\phi_-(q)$ and part of the 3-dissection of $A(-q)$.
Following Hickerson and Mortenson \cite{Hi-Mo14} we define
$$
j(z;q):=(z;q)_\infty (z^{-1}q;q)_\infty (q;q)_\infty,
$$
and 
$$
m(x,q,z):=\frac{1}{j(z;q)}\sum_{r=-\infty}^{\infty}
          \frac{(-1)^rq^{r(r-1)/2}z^r}{1-q^{r-1}xz}.
$$
Hickerson and Mortenson used $m(x,q,z)$ as a building block for expressing
the mock theta functions in terms of Appell-Lerch sums.  
They found
\beq
\label{Aqm}
A(q)=-m(q,q^4,q^2), \qquad\mbox{(\cite[Eq.(5.1),p.399]{Hi-Mo14})}
\eeq
and
\beq
\label{phi-qm}
\phi_-(q)=-m(q,q^3,q)-q\frac{J_6^6}{J_2^2J_3^3},
\qquad\mbox{(\cite[Eq.(5.30),p.401]{Hi-Mo14}.)}
\eeq
We define the usual Atkin $U_p$ operator which acts on a formal power series
$$
f(q)=\sum_{n\in \mathbb{Z}}a(n)q^n,
$$
by                                 
$$
\Up{p}{f(q)}=\sum_{n\in \mathbb{Z}}a(pn)q^n.
$$

\begin{lemma}
\label{lemphi}
We have
\beq
\label{lmphi}
\Up{3}{A(-q)}=\phi_{-}(q)-4q\frac{J_6^2J_4^8J_1^2}{J_3J_2^{10}},
\eeq
so that 
$$
\Up{3}{A(-q)}\equiv \phi_{-}(q)\pmod4.
$$
\end{lemma}

\begin{proof}
Replacing $q$ by $q^4$ and setting $x=-q$ in \cite[Corollary 3.8]{Hi-Mo14}, we have
\begin{align}
\label{mq1}
m(-q,q^4,-1)=&m(-q^{15},q^{36},-1)+\frac{1}{q^3}m(-q^3,q^{36},-1)\\
&+\frac{1}{q^{10}}m(-q^{-9},q^{36},-1)-\frac{1}{2q^3}\frac{J_{36}J_{24}J_{12}J_{6}J_{2}}{J_{72}^2J_{8}J_{3}},
\nonumber
\end{align}
after simplification. Replacing $q$ by $q^4$ and setting $x=-q$, $z_0=q^2$ and $z_1=-1$ in \cite[Theorem 3.3]{Hi-Mo14}, we have
\beq
\label{mq2}
m(-q,q^4,-1)=m(-q,q^4,q^2)+\frac{J_4^{10}J_1}{2J_8^4J_2^6},
\eeq
after simplification. Hence by \myeqref{Aqm}, \myeqref{mq1} and \myeqref{mq2}
\beq
\label{U3A0}
\Up{3}{-A(-q)}=m(-q^5,q^{12},-1)+\frac{1}{q}m(-q,q^{12},-1)-\Up{3}{G(q)},
\eeq
where
$$
G(q)=\frac{1}{2}
\left(\frac{J_{36}J_{24}J_{12}J_{6}J_{2}}{q^3J_{72}^2J_{8}J_{3}}
     +\frac{J_4^{10}J_1}{J_8^4J_2^6}\right).
$$
The following identities may be proved using the theory of modular functions
by verifying that both sides hold for a sufficient number of terms:
$$                   
\Up{3}{\frac{J_{24}^3J_{2}}{J_{36}J_{18}J_{12}J_8}}=1,\qquad
\Up{3}{\frac{J_{24}^4J_4^{10}J_9J_6^{12}J_1}{J_{18}^2J_{12}^{10}J_8^4J_3^6J_2^6}}=1.
$$                   
This verification was carried out using the second author's 
maple \texttt{ETA} package,
which is available at
\begin{center}
\url{https://qseries.org/fgarvan/qmaple/ETA/}
\end{center}
The two functions involved can be rewritten in terms of eta-products,
which turn out to be modular functions on $\Gamma_0(72)$. We then deduce that
\beq         
\Up{3}{\frac{J_{36}J_{24}J_{12}J_{6}J_{2}}{q^3J_{72}^2J_{8}J_{3}}}=\frac{J_{12}^2J_6J_4^2J_2}{qJ_{24}^2J_8^2J_1},\qquad
\Up{3}{\frac{J_4^{10}J_1}{J_8^4J_2^6}}=\frac{J_6^2J_4^{10}J_1^6}{J_8^4J_3J_2^{12}}.
\label{mpG}
\eeq
Substituting \myeqref{mpG} into \myeqref{U3A0}, we have
\beq
\label{U3A}
\Up{3}{-A(-q)}=m(-q^5,q^{12},-1)+\frac{1}{q}m(-q,q^{12},-1)-\frac{J_{12}^2J_6J_4^2J_2}{2qJ_{24}^2J_8^2J_1}-\frac{J_6^2J_4^{10}J_1^6}{2J_8^4J_3J_2^{12}},
\eeq
Replacing $q$ by $q^3$ and setting $x=z=q^2$ in \cite[Corollary 3.7]{Hi-Mo14}, we have
\beq
\label{mqp1}
m(q^2,q^3,q^2)=m(-q^7,q^{12},q^8)-\frac{1}{q}m(-q,q^{12},q^8)+2q\frac{J_{12}^3J_3^3J_2}{J_6^3J_4J_1^2},
\eeq
after simplification. Since the equations in \cite[Proposition 3.1]{Hi-Mo14} imply that
$$
m(x,q,z)+m(q/x,q,q/z)=1,
$$
we can rewrite \myeqref{mqp1} as
\beq
\label{mqp2}
m(q,q^3,q)=m(-q^5,q^{12},q^4)+\frac{1}{q}m(-q,q^{12},q^8)-2q\frac{J_{12}^3J_3^3J_2}{J_6^3J_4J_1^2}.
\eeq
Again using \cite[Theorem 3.3]{Hi-Mo14} we have
\begin{align}
\label{mqp3}
m(-q^5,q^{12},-1)&=m(-q^5,q^{12},q^4)+P_1(q),\\
\label{mqp4}
m(-q,q^{12},-1)&=m(-q,q^{12},q^8)+P_2(q),
\end{align}
where
\begin{align*}
P_1(q)=&\frac{J_{12}^3j(-q^{4};q^{12})j(q^9;q^{12})}
             {j(q^4;q^{12})j(-1;q^{12})j(-q^9;q^{12})j(q^5;q^{12})},\\
P_2(q)=&\frac{J_{12}^3j(-q^{4};q^{12})j(q^9;q^{12})}
             {j(q^4;q^{12})j(-1;q^{12})j(-q^9;q^{12})j(q;q^{12})}.
\end{align*}
We have
\beq
P_1(q) + \frac{1}{q} P_2(q) =
\frac{ J_{12}^3 j(-q^4,q^{12}) j(q^9,q^{12})}
{q j(q^4,q^{12}) j(-1,q^{12}) j(-q^9,q^{12}) j(q^5,q^{12}) j(q,q^{12})}
\Big(j(q^5,q^{12}) + qj(q,q^{12}\Big).
\mylabel{eq:P12ida}
\eeq
In the quintuple product identity \myeqn{quintprod} we let $z=q$
and replace $q$ by $q^4$ to find that
\beq
j(q^5,q^{12}) + qj(q,q^{12}) = (-q;q^2)_\infty (q^2;q^2)_\infty.
\mylabel{eq:quinpapp}
\eeq
Equations \myeqn{P12ida} and \myeqn{quinpapp} give the following identity
\beq
\label{mqp5}
P_1(q)+\frac{1}{q}P_2(q)=\frac{J_{12}^5J_8J_3^3J_2^4}{2qJ_{24}^3J_6^4J_4^3J_1^2}
\eeq
after some simplification using Jacobi's triple product identity \myeqref{jacp}.
Combining \myeqref{phi-qm}, \myeqref{U3A} and \myeqref{mqp2}-\myeqref{mqp5}, we have
\beq
\label{lmphi1}
\Up{3}{A(-q)}=\phi_{-}(q)+\frac{J_{12}^2J_6J_4^2J_2}{2qJ_{24}^2J_8^2J_1}+\frac{J_6^2J_4^{10}J_1^6}{2J_8^4J_3J_2^{12}}+\frac{qJ_6^6}{J_3^3J_2^2}-\frac{J_{12}^5J_8J_3^3J_2^4}{2qJ_{24}^3J_6^4J_4^3J_1^2}-\frac{2qJ_{12}^3J_3^3J_2}{J_6^3J_4J_1^2}.
\eeq
Finally we have 
\beq
\label{lmphi2}
\frac{J_{12}^2J_6J_4^2J_2}{2qJ_{24}^2J_8^2J_1}+\frac{J_6^2J_4^{10}J_1^6}{2J_8^4J_3J_2^{12}}+q\frac{J_6^6}{J_3^3J_2^2}-\frac{J_{12}^5J_8J_3^3J_2^4}{2qJ_{24}^3J_6^4J_4^3J_1^2}-2q\frac{J_{12}^3J_3^3J_2}{J_6^3J_4J_1^2}=-4q\frac{J_6^2J_4^8J_1^2}{J_3J_2^{10}}.
\eeq
This can be proved using the maple \texttt{ETA} package. The identity can be rewritten
as an identity for eta-products that  are modular functions on $\Gamma_0(24)$.
Equation \myeqref{lmphi} follows from \myeqref{lmphi1} and \myeqref{lmphi2}.
The mod $4$ congruence follows immediately.
\end{proof}

Lemmas \myref{lem:NAH} and \myref{lemphi}
imply the following
\begin{lemma}
\label{lem:NphiH}
Let $n$ be a positive integer.
\beq
\label{eq:NphiH}
N_{\phi_-}(n)\equiv -H(24n-1)\pmod4,
\eeq
\end{lemma}

We note that if $(m,6)=1$ then $m^2\equiv1\pmod{24}$.
Thus the following theorem follows easily from Lemma \myref{lem:NphiH} and
Theorem \myref{thm:the12}.
\begin{theorem}
\label{the12phi-1}
For $n>0$ be an integer,
\begin{enumerate}
\item[(1)] $N_{\phi_-}(n)$ is odd if and only if $24n-1$ has the form
$$
24n-1=p^{4a+1}m^2,
$$
where $p$ is prime, and $m$ and $a$ are integers satisfying
$(m,p)=1$ and $a\ge0$.
\item[(2)]  $N_{\phi_-}(n)\equiv 2\pmod4$ if and only if $24n-1$ has the form
$$
24n-1=p_1^{4a+1}p_2^{4b+1}m^2,
$$
where $p_1$ and $p_2$ are primes such that $\displaystyle\leg{p_1}{p_2}=-1$,
 $(m,p_1p_2)=1$ and $a,b\geq 0$ are integers.
\end{enumerate}
\begin{remark}
Wang \cite[Theorem 6.6]{Wa2020} proved (1) by showing that
$$
N_{\phi_-}(n) \equiv N_{\psi}(n) \pmod{2}.
$$
\end{remark}
\end{theorem}

\begin{theorem}
\label{thm:Nphimmod4}
Let $p>3$ be a prime and $p\not\equiv23\pmod{24}$. Suppose $24\delta_p\equiv1\pmod{p^2}$ and $k,n\in \mathbb{Z}$ where $\displaystyle\leg{k}{p}=1$. Then
$$
N_{\phi_-}(p^2n+(pk+1)\delta_p)\equiv 0\pmod4.
$$
\end{theorem}

\begin{proof}
Let $p\not\equiv23\pmod{24}$ be a prime $> 3$,
and suppose $n$, $k$ are nonnegative integers where $\displaystyle\leg{k}{p}=1$. 
Suppose
$m=p^2n+(pk+1)\delta_p$, where $\delta_p=\tfrac{1}{24}(23p^2+1)$.
Since
\beq
\label{24m-1}
24m-1\equiv pk\pmod{p^2},
\eeq
we have $p\|24m-1$. Theorem \myref{the12phi-1}(1) implies that $N_{\phi_{-}}(m)$ 
is even
since $p\not\equiv23\pmod{24}$.
Now suppose $N_{\phi_{-}}(m)\equiv 2\pmod4$. Then Theorem \myref{the12phi-1}(2) 
and \myeqref{24m-1} imply that 
$$
24m-1=p^1 q^{4b+1}t^2,\quad\mbox{and}\quad
k\equiv q^{4b+1}t^2\pmod{p},
$$
where $q$ is a prime satisfying 
$\displaystyle\leg{p}{q}=-1$, 
$(pq,t)=1$, and $b\ge0$, $t>0$ are integers. 
Since $pq\equiv-1\pmod{24}$
either $p$ or $q\equiv1\pmod{4}$
so by quadratic reciprocity $\leg{q}{p}=\leg{p}{q}=-1$.
But
$$
\leg{q}{p}=\leg{q^{4b+1}t^2}{p}=\leg{k}{p}=1,
$$
which is a contradiction.
Hence $N_{\phi_{-}}(m)\not\equiv 2\pmod4$, and       
$$
N_{\phi_{-}}(m)\equiv 0\pmod4.
$$
\end{proof}

\subsection{The sixth order mock theta function $\sigma(q)$}
Ramanujan's sixth order mock theta function
$$
\sigma(q):=\sum_{n=0}^\infty N_\sigma(n) q^n :=
\sum_{n=0}^{\infty}\frac{q^{(n+1)(n+2)/2}(-q;q)_n}{(q;q^2)_{n+1}}
$$
was first studied by Andrews and Hickerson \cite{An-Hi91}.
By \cite[Theorem 1.1]{Be-Ch07} we have
\beq
\label{sigmaq}
\sigma(q)=\phi_{-}(q^2)+qP_3(q^2),
\eeq
where
\beq
\label{Pq2}
P_3(q)=(-q;q)_\infty^2(-q^3,-q^3,q^3;q^3)_\infty=\frac{J_6^2J_2^2}{J_3J_1^2}.
\eeq
The following lemma follows from Lemma \myref{lem:NphiH} together with
equations \myeqref{sigmaq} and \myeqref{p11}.

\begin{lemma}
\label{lem:NsigmaH}
For each integer $n>0$, we have
\beq
\label{NsigmaH}
N_{\sigma}(n)\equiv (-1)^{n+1}H(12n-1)\pmod4.
\eeq
\end{lemma}

\begin{proof}
By Lemma \myref{lem:NphiH} and equation \myeqref{sigmaq} we have
$$
N_{\sigma}(2n)=N_{\phi_{-}}(n)\equiv -H(24n-1)\pmod4,
$$
which implies \myeqref{NsigmaH} for even $n$.

By \myeqref{eq:thetamod4}, \myeqref{p11} and \myeqref{Pq2} we have
$$
P_3(q)=\frac{J_6^2J_2^2}{J_3J_1^2}
=\frac{J_3^3J_2^2}{J_1^2}\cdot\Parans{\frac{J_6}{J_3^2}}^2\equiv \frac{J_3^3J_2^2}{J_1^2}
=\mathscr{H}_{24,11}(q)\pmod4.
$$
Hence
$$
N_{\sigma}(2n+1)\equiv H(24n+11)\pmod4,
$$
and \myeqref{NsigmaH} holds for odd $n$.
\end{proof}

Again we note that if $(m,6)=1$ then $m^2\equiv1\pmod{12}$.
Thus the following theorem follows easily from Lemma \myref{lem:NsigmaH} and
Theorem \myref{thm:the12}.

\begin{theorem}
\label{the12sigma}
For $n>0$ be an integer,
\begin{enumerate}
\item[(1)] $N_{\sigma}(n)$ is odd if and only if $12n-1$ has the form
$$
12n-1=p^{4a+1}m^2,
$$
where $p$ is prime, and $m$ and $a$ are integers satisfying
$(m,p)=1$ and $a\ge0$.
\item[(2)]  $N_{\sigma}(n)\equiv 2\pmod4$ if and only if $12n-1$ has the form
$$
12n-1=p_1^{4a+1}p_2^{4b+1}m^2,
$$
where $p_1$ and $p_2$ are primes such that $\displaystyle\leg{p_1}{p_2}=-1$,
 $(m,p_1p_2)=1$ and $a,b\geq 0$ are integers.
\end{enumerate}
\end{theorem}
\begin{remark}
Wang \cite[Theorem 6.4]{Wa2020} proved (1) by using a Hecke-Rogers series
identity of Andrews and Hickerson \cite{An-Hi91}.
\end{remark}

\begin{theorem}
\label{thm:Nsigmamod4}
Let $p>3$ be a prime and $p\not\equiv11\pmod{12}$. 
Suppose $12\delta_p\equiv1\pmod{p^2}$ and $k,n\in \mathbb{Z}$ where 
$\displaystyle\leg{k}{p}=1$. Then
$$
N_{\sigma}(p^2n+(pk+1)\delta_p)\equiv 0\pmod4.
$$
\end{theorem}

\begin{proof}
Let $p\not\equiv11\pmod{12}$ be a prime $> 3$,
and suppose $n$, $k$ are nonnegative integers where $\displaystyle\leg{k}{p}=1$. 
Suppose
$m=p^2n+(pk+1)\delta_p$, where $\delta_p=\tfrac{1}{12}(11p^2+1)$.
Since
\beq
\label{12m-1}
12m-1\equiv pk\pmod{p^2},
\eeq
we have $p\|12m-1$. Theorem \myref{the12sigma}(1) implies that $N_{\sigma}(m)$ 
is even
since $p\not\equiv23\pmod{24}$.
Now suppose $N_{\sigma}(m)\equiv 2\pmod4$. Then Theorem \myref{the12sigma}(2) 
and \myeqref{12m-1} imply that 
$$
12m-1=p^1 q^{4b+1}t^2,\quad\mbox{and}\quad
k\equiv q^{4b+1}t^2\pmod{p},
$$
where $q$ is a prime satisfying 
$\displaystyle\leg{p}{q}=-1$, 
$(pq,t)=1$, and $b\ge0$, $t>0$ are integers. 
Since $pq\equiv-1\pmod{12}$
either $p$ or $q\equiv1\pmod{4}$
so by quadratic reciprocity $\leg{q}{p}=\leg{p}{q}=-1$.
However
$$
\leg{q}{p}=\leg{q^{4b+1}t^2}{p}=\leg{k}{p}=1,
$$
which is a contradiction.
Hence $N_{\sigma}(m)\not\equiv 2\pmod4$, and       
$$
N_{\sigma}(m)\equiv 0\pmod4.
$$
\end{proof}

\section{Mod $4$ conjectures for unimodal sequences and the Andrews spt-function}
\label{sec:seq}

In Section \myref{sec:mock} we  derived mod $4$ congruences for four mock theta 
functions that are  closely associated with the Hurwitz class number. 
In this section we prove the corrected version of Kim, Lim and Lovejoy's mod $4$
conjectures for odd-balanced unimodal sequences, the second author's
mod $4$ conjectures for the Andrews's spt-function and
Bryson, Ono, Pitman, and Rhoades's mod $4$ conjectures for
strongly unimodal sequences. The proofs depend on some of the mod $4$ congruences
for certain mock theta functions in the previous section, as well as some
new Hecke-Rogers identities.

\subsection{Number of partitions in odd-balanced unimodal sequences}
\label{subsec:oddbal}
Recall from Section \myref{sec:intro} that
\beq
\label{eq:Vzqdef}
\mathcal{V}(z;q) := \sum_{m,n} v(m,n) z^m q^n 
= \sum_{n=0}^\infty \frac{(-zq;q)_n (-z^{-1}q;q)_n q^{n}}{(q;q^2)_{n+1}}.
\eeq
is the two-parameter generating function for odd-balanced unimodal
sequences of size $2n+2$ and rank $m$.

By \cite[p.258]{Mo2014} we have
\begin{align}
\label{eq:Vzqid}
\left(1+\frac{1}{z}\right)\mathcal{V}(z,q)=&\frac{(-q;q)_\infty}{(q;q)_\infty}\left(\sum_{n,r\geq0}-\sum_{n,r<0}\right)(-1)^nz^rq^{n^2+2n+(2n+1)r+r(r+1)/2}\\
=&\frac{(-q;q)_\infty}{(q;q)_\infty}\sum_{n,r\geq0}(-1)^n(z^r+z^{-r-1})q^{n^2+2n+(2n+1)r+r(r+1)/2}
\nonumber\\
=&\frac{(-q;q)_\infty}{(q;q)_\infty}\sum_{n,r\geq0}(-1)^n(z^r+z^{-r-1})q^{(n+r)^2+2(n+r)-r(r+1)/2}
\nonumber\\
=&\frac{(-q;q)_\infty}{(q;q)_\infty}\sum_{n}^{\infty}\sum_{r=0}^{n}(-1)^{n+r}(z^r+z^{-r-1})q^{n^2+2n-r(r+1)/2}.
\nonumber
\end{align}
See also \cite[Eq (1.12)]{Ki-Li-Lo16}.            
Hence letting $z=1$ we find that
\beq
\label{VHR}
\sum_{n=0}^{\infty}v(n)q^n=\mathcal{V}(1,q)=\frac{(-q;q)_\infty}{(q;q)_\infty}\sum_{n=0}^{\infty}\sum_{r=0}^{n}(-1)^{n+r}q^{n^2+2n-r(r+1)/2},
\eeq
where $v(n)$ is the number of odd-balanced unimodal sequences 
of size $2n+2$. Letting $z=i$ in \myeqref{eq:Vzqdef} and \myeqref{eq:Vzqid}
we have
\beq
\label{AHR} 
\frac{A(q)}{q}=\mathcal{V}(i,q)
=\frac{(-q;q)_\infty}{(q;q)_\infty}
\sum_{n=0}^{\infty}\sum_{r=0}^{n}(-1)^{n+r+r(r+1)/2}q^{n^2+2n-r(r+1)/2},
\eeq
by \myeqref{Aqdef}, after some simplification.
The form of the series in equations \myeqref{VHR} and \myeqref{AHR} 
is quite similar. It is easy see that
$$
v(n)\equiv N_A(n+1)\pmod2.
$$
As expected Theorem \myref{the12A}(1) confirms the equivalent 
\cite[Theorem 1.1]{Ki-Li-Lo16}. We now consider $v(n)$ mod $4$.
In this section we prove
\begin{theorem}
\label{thev4}
Let $p\not\equiv\pm1\pmod8$ be an odd prime, suppose $8\delta_p\equiv1\pmod{p^2}$ and $k,n\in \mathbb{Z}$ where $\displaystyle\leg{k}{p}=1$. Then
$$
v(p^2n+(pk-7)\delta_p)\equiv 0\pmod4.
$$
\end{theorem}
\begin{remark}
We have corrected Kim, Lim and Lovejoy's original Conjecture \myref{conjv}.
\end{remark}
We define 
$$
D_v(q):=\sum_{n=0}^{\infty}d_v(n)q^n
:=\sum_{m=0}^{\infty}\sum_{r=0}^{m}q^{m^2+2m-r(r+1)/2}.
$$

\begin{lemma}
\label{lemdv}
If $p\equiv 3,5\pmod8$ is prime and $p\|8n+7$ then $d_v(n)=0$.
\end{lemma}

\begin{proof}
Suppose that $p\equiv 3,5\pmod8$ is prime and $p\|8n+7$. 
Suppose by way of contradiction that
$d_v(n)\ne0$. Then
$$ 
8n+7= 8(m^2+2m - r(r+1)/2) + 7 = 8(m+1)^2-(2r+1)^2,
$$
for some integers $m\ge0$ and  $0\leq r\leq m$. Since $p\mid 8n+7$ this implies 
$$
8(m+1)^2\equiv(2r+1)^2\pmod{p}.
$$
Since  $p\equiv 3,5\pmod8$,
$\displaystyle\leg{8}{p}=-1$ and $(m+1)\equiv (2r+1)\equiv 0 \pmod{p}$.
But this implies $p^2 \mid 8n+7$, which contradicts $p\| 8n+7$.
We conclude that $d_v(n)=0$.
\end{proof}

\begin{proof}[Proof of Theorem \myref{thev4}]
Let
$$
D_0(q):=
\mathcal{V}(1,q)-\frac{A(q)}{q}=\sum_{n=0}^{\infty}(v(n)-N_A(n+1))q^n.
$$
By \myeqref{VHR} and \myeqref{AHR} we have 
\begin{align*}
D_0(q)&=\frac{(-q;q)_\infty}{(q;q)_\infty}
         \sum_{n=0}^{\infty}
         \sum_{r=0}^{n}(-1)^{n+r}(1-(-1)^{r(r+1)/2})q^{n^2+2n-r(r+1)/2}\\
&=2\frac{(-q;q)_\infty}{(q;q)_\infty}
      \sum_{n=0}^{\infty}\sum_{r=0}^{n}(-1)^{n+r}\epsilon(r)q^{n^2+2n-r(r+1)/2},
\end{align*}
where $\epsilon(r)=1$ if $r\equiv 1,2\pmod4$ and $\epsilon(r)=0$ otherwise. 
We recall \myeqref{eq:thetamod4} and note that
$$
\frac{(q;q)_\infty}{(-q;q)_\infty} = \frac{J_1^2}{J_2}
=1+2\sum_{n=1}^{\infty}(-1)^nq^{n^2}\equiv 1\pmod2.
$$
Therefore we have
\beq
\label{d0qv}
D_0(q)  \equiv 
  2\sum_{n=0}^{\infty}
   \sum_{r=0}^{n}(-1)^{n+r}\epsilon(r)q^{n^2+2n-r(r+1)/2}\pmod4.
\eeq
Let
$$
\sum_{n=0}^{\infty}d(n)q^n:=\sum_{n=0}^{\infty}\sum_{r=0}^{n}(-1)^{n+r}\epsilon(r)q^{n^2+2n-r(r+1)/2}.
$$
   Now assume $p\not\equiv\pm1\pmod{8}$ is an odd prime,  
and let $n$, $k$ be  integers where $\displaystyle\leg{k}{p}=1$. 
Suppose
$m=p^2n+(pk+1)\delta_p$, where $\delta_p=\tfrac{1}{8}(7p^2+1)$.
Then
$$
8m + 7 \equiv pk \pmod{p^2},
$$
which implies that $p\|8m+7$.  Since $p\equiv 3,5\pmod{8}$, Lemma \myref{lemdv}
implies $d_v(m)=0$.  But $d_v(m)=0$ implies $d(m)=0$ and
$$
N_A(m+1) \equiv v(m) \pmod{4},
$$
by \myeqref{d0qv}. Now
$$
m+1 \equiv (pk+1) \delta_p - 8\delta_p + 1 \equiv (pk+1)\delta_p \pmod{p^2}.
$$
So $N_A(m+1)\equiv 0 \pmod{4}$ by Theorem \myref{theA4} and hence $v(m)\equiv 0 \pmod4$.
\end{proof}

\subsection{The Andrews' $\mathrm{spt}(n)$ function.}
\label{subsec:spt}

By equation \myeqref{eq:sptPsimod4} and \cite[Theorem 1.3]{An-Ga-Li13} we have
the following theorem.
\begin{theorem}
\label{thepsi1}
For each $n>0$, $N_{\psi}(n)$ is odd if and only if 
$$
24n-1=p^{4a+1}m^2,
$$
for some prime $p$, and some integers $a$, $m$
satisfying
$(m,p)=1$ and $a\ge0$.
\end{theorem}

Berndt and Chan \cite[p.776]{Be-Ch07} found a Hecke-Rogers identity
for the sixth order mock theta function $\phi_{-}(q)$.
\beq
\label{phi-hec}
\frac{J_1^2}{J_2}\phi_{-}(q)
=\sum_{n=1}^{\infty}\sum_{m=1-n}^{n}(-1)^{n-1}q^{n(3n-1)-2m^2+m}(1-q^{2n}).
\eeq
We need a similar Hecke-Rogers identity for $\psi(q)$.

Andrews \cite[Eq. (1.10)]{An2012}, Mortenson \cite[Eq. (2.5)]{Mo2013}, 
Chen and Wang \cite[Eq. (4.37)]{Ch-Wa2020} have  found Hecke-Rogers series
for $\psi(q)$. The second author \cite{Ga15a} has shown how to express 
each of Ramanujan's third order mock theta functions in terms of Hecke-Rogers
series. We need a new additional  Hecke-Rogers series for $\psi(q)$. 

Using Jacobi's triple product
identity \myeqref{jacp} we find that
\beq
\label{eq:trizid}
\sum_{n=-\infty}^\infty (-1)^n q^{n(n+1)/2 + bn} =0
\eeq
for any integer $b$. Similarly for any integer $b$ we have
\beq
\label{eq:pentid}
\sum_{n=-\infty}^\infty (-1)^n q^{n(3n+1)/2 + bn} =
\begin{cases}
(-1)^b q^{-b(b+1)/2} (q)_\infty, & \mbox{if $b\equiv 0\pmod{3}$},\\
0, & \mbox{if $b\equiv 1\pmod{3}$},\\
(-1)^{b-1} q^{-b(b+1)/2} (q)_\infty, & \mbox{if $b\equiv 2\pmod{3}$}.
\end{cases}
\eeq
See \cite[p.99]{At-SwD}.

\begin{lemma}
\label{lem:psiHR}
We have
\beq
\label{psihec}
\frac{J_1^2}{J_2}\psi(q)
=\sum_{n=1}^{\infty}\sum_{m=1-n}^{n}(-1)^{m-1}q^{n(3n-1)-2m^2+m}(1-q^{2n}).
\eeq
\end{lemma}
\begin{remark}
We note that this Hecke-Rogers series for $\psi(q)$ may be
deduced from a general theorem of Bradley-Thrush \cite{BrTh2020}.
In fact a two-parameter generalization may be obtained from letting
$k=2$, $p=q^6$ and $x=q^{-2}$ in \cite[Theorem 7.3]{BrTh2020}.
Our proof is different.  This method will be used to obtain
a new Hecke-Rogers series for $U(q)$, the generating function for
strongly unimodal sequences. See Lemma \myref{lem:newUHR} below.
\end{remark}

\begin{proof}
From \cite[Eq. (4.37)]{Ch-Wa2020} we have 
\beq
\label{eq:J1psiid}
J_1\psi(q)
=\sum_{n=1}^{\infty}\sum_{r=1}^{n}(-1)^{n-1}q^{2n^2-n-r(r-1)/2}(1-q^{2n}).
\eeq
We note that \myeqref{eq:J1psiid} also follows from 
\cite[Eq(1.3),Eq(1.5)]{Ki-Li-Lo16}.             
We define
$$
A_k:=
\sum_{n=1}^{\infty}\sum_{r=1}^{n}(-1)^{n-1}q^{2n^2-n-(r-k)(r-k-1)/2}(1-q^{2n}),
$$
so that $A_0=J_1\psi(q)$.  Using \myeqref{eq:pentid}
we derive a recurrence relation for $A_k$.
\begin{align*}
&A_k+A_{-k}-(A_{k-1}+A_{1-k})\\
&=\sum_{n=1}^{\infty}
   \sum_{r=1}^{n}(-1)^{n-1}(q^{2n^2-n-(r-k)(r-k-1)/2}
      +q^{2n^2-n-(r+k)(r+k-1)/2})(1-q^{2n})\\
&-\sum_{n=1}^{\infty}
   \sum_{r=1}^{n}(-1)^{n-1}(q^{2n^2-n-(r-k+1)(r-k)/2}
      +q^{2n^2-n-(r+k-1)(r+k-2)/2})(1-q^{2n})\\
&=\sum_{n=1}^{\infty}(-1)^n(q^{2n^2-n-(n-k)(n-k+1)/2}
   -q^{2n^2-n-(n+k)(n+k-1)/2})(1-q^{2n})\\
&=\sum_{n=-\infty}^\infty (-1)^n q^{ n(3n+1)/2 + n(1-k) - k(k-1)/2}
  - \sum_{n=-\infty}^\infty (-1)^n q^{ n(3n+1)/2 + nk - k(k-1)/2}\\
&=-J_1a_k,
\end{align*}
where
\begin{align*}
&a_{3r}:=(-1)^rq^{-6r^2+r},\\
&a_{3r+1}:=(-1)^{r-1}q^{-6r^2-r},\\
&a_{3r+2}:=(-1)^{r-1}(q^{-6r^2-5r-1}+q^{-6r^2-7r-2}).
\end{align*}
Therefore for $k>0$ we have
\beq
\label{iak-k}
A_k+A_{-k}=2A_0-J_1\sum_{r=1}^{k}a_r.
\eeq
By \myeqref{eq:trizid} we find that
\beq
\label{iAk01}
\sum_{k=-\infty}^{\infty}(-1)^k A_k q^{k^2}=
 \sum_{n=1}^{\infty}\sum_{r=1}^{n}(-1)^{n-1}q^{2n^2-n-r(r-1)/2}(1-q^{2n})
 \sum_{k=-\infty}^\infty (-1)^k q^{k(k-1)/2 + rk} = 0.
\eeq
By \myeqref{iak-k}
\beq
\label{iAk02}
\sum_{k=-\infty}^{\infty}(-1)^kA_kq^{k^2}=\frac{J_1^2}{J_2}A_0-J_1\sum_{k=1}^{\infty}\sum_{r=1}^{k}(-1)^kq^{k^2}a_r.
\eeq
By \myeqref{iAk01} and \myeqref{iAk02} we have
\beq
\label{psih}
\frac{J_1^2}{J_2}A_0=J_1\sum_{k=1}^{\infty}\sum_{r=1}^{k}(-1)^kq^{k^2}a_r.
\eeq
Let $G(r,k):=(-1)^kq^{k^2}a_r$ and 
$F(m,n):=\sgtwid(n) (-1)^{m-1}q^{n(3n-1)-2m^2+m}$, 
where $\sgtwid(n)=1$ for $n>0$,
$\sgtwid(0)=0$ and $\sgtwid(n)=-1$ for $n<0$.
It is easy to check that
$$                
G(3s+t,k) = \begin{cases}
F(3s-k,2s-k),& \mbox{if $t=0$,}\\
F(k-3s,k-2s), & \mbox{if $t=1$,}\\
F(3s+2-k,k-2s-1)+F(k-3s-1,2s-k+1) & \mbox{if $t=2$,}
\end{cases}
$$             
assuming $0 < 3s+t \le k$.
We find that
\beq
\label{gf1}
\sum_{k=1}^\infty \sum_{0<3s\leq k}G(3s,k)
=\sum_{k=1}^\infty \sum_{0\leq 3s\leq k}F(3s-k,2s-k)
=\sum_{n=-\infty}^{-1}\sum_{m=1+n}^{0}F(m,n).
\eeq
Similarly we find that
\beq
\label{gf2}
\sum_{k=1}^\infty \sum_{0<3s+1\leq k}G(3s+1,k)
=\sum_{n=1}^{\infty}\sum_{m=1}^{n}F(m,n),
\eeq
and
\beq
\label{gf3}
\sum_{k=1}^\infty \sum_{0<3s+2\leq k}G(3s+2,k)
=\left(\sum_{n=1}^{\infty}\sum_{m=1-n}^{0}
  +\sum_{n=-\infty}^{-1}\sum_{m=1}^{-n}\right)F(m,n).
\eeq
Hence by \myeqref{psih} - \myeqref{gf3}
\begin{align}
\label{psih1}
\frac{J_1^2}{J_2}\psi(q)=&\sum_{k=1}^{\infty}\sum_{r=1}^{k}G(r,k)
=\sum_{n=-\infty}^\infty\sum_{m=1-\abs{n}}^{\abs{n}}F(m,n)\\
\nonumber
=&\sum_{n=1}^{\infty}\sum_{m=1-n}^{n}(-1)^{m-1}q^{n(3n-1)-2m^2+m}(1-q^{2n}).
\end{align}
\end{proof}

We define the following three series
\begin{align}
\label{D0q}
D_0(q):=&\sum_{n=1}^{\infty}d_0(n)q^n:=
\sum_{n=-\infty}^\infty
\sum_{1-\abs{n}\leq m\leq\abs{n}}q^{n(3n-1)-m(2m-1)},\\
\nonumber
D_1(q):=&\sum_{n=1}^{\infty}d_1(n)q^n:=
\sum_{n=-\infty}^\infty
\sum_{1-\abs{n}\leq m\leq\abs{n}}(1-\epsilon(m,n))q^{n(3n-1)-m(2m-1)},\\
\nonumber
D_2(q):=&\sum_{n=1}^{\infty}d_2(n)q^n
:=
\sum_{n=-\infty}^\infty
\sum_{1-\abs{n}\leq m\leq\abs{n}}\epsilon(m,n)q^{n(3n-1)-m(2m-1)},
\end{align}
where $\epsilon(m,n)=1$ if $m\equiv n\pmod2$ and $\epsilon(m,n)=0$ otherwise. 
Clearly $D_0(q)=D_1(q)+D_2(q)$.

\begin{lemma}
\label{lemD1}
We have
$$
\phi_{-}(q)-\psi(q)\equiv 2D_1(q)\pmod4.
$$
\end{lemma}

\begin{proof}
By \myeqref{phi-hec} and \myeqref{psihec}
\begin{align*}
\phi_{-}(q)-\psi(q)=&\frac{J_2}{J_1^2}\sum_{n=1}^{\infty}\sum_{m=1-n}^{n}((-1)^{n-1}-(-1)^{m-1})q^{n(3n-1)-2m^2+m}(1-q^{2n})\\
\equiv &2\sum_{n=1}^{\infty}\sum_{j=1-n}^{n}\frac{(-1)^{n-1}-(-1)^{m-1}}{2}q^{n(3n-1)-2m^2+m}(1+q^{2n})\\
\equiv &2\sum_{n=-\infty}^\infty\sum_{1-\abs{n}\leq m\leq\abs{n}}(1-\epsilon(m,n))q^{n(3n-1)-2m^2+m}\\
\equiv &2D_1(q)\pmod 4.
\end{align*}
\end{proof}

We note that Lemma \myref{lemD1} implies that $N_{\phi_{-}}(n)$ and $N_{\psi}(n)$
have the same parity. This is confirmed by Theorem \myref{the12phi-1}(1) and
Theorem \myref{thepsi1}. The following two lemmas are related to solutions of the
Pell equation
\beq
\label{eq:Pell}
u^2 - 6v^2 = m,
\eeq
where $u$, $v$, $m$ are integers.
Following \cite{An-Dy-Hi88} we say two solutions $(u,v)$ and $(u',v')$ are
\textbf{equivalent} if 
\beq
\label{eq:Pellequiv}
u' + v'\sqrt{6} = \pm(5+2\sqrt{6})^r (u+v\sqrt{6}),
\eeq
for some integer $r$. By \cite[Lemma 3]{An-Dy-Hi88}, if $m>0$, then
each equivalence class of solutions of \myeqref{eq:Pell} contains a unique
$(u,v)$ with $u>0$ and 
\beq
\label{eq:uniquerep}
-\tfrac{1}{3} u < v \le \tfrac{1}{3}u.
\eeq
We define $H_{\sqrt{6}}(m)$ to be the number of inequivalent solutions
to \myeqref{eq:Pell}. By \myeqref{eq:uniquerep} we have
\beq
\label{eq:H6gen}
\sum_{m=1}^\infty H_{\sqrt{6}}(m) q^m = \sum_{u=1}^\infty \sum_{-\tfrac{1}{3} u < v \le \tfrac{1}{3}u} q^{u^2 - 6v^2}.
\eeq
Lovejoy \cite{Lo2004} has calculated $H_{\sqrt{6}}(m)$. 
Wang \cite[Lemma 2.7]{Wa2020} extended this to negative $m$.
\begin{lemma}[{Lovejoy \cite[Theorem 1.3]{Lo2004}}]
\label{lem:lovejoy6}
Let $m$ have prime factorization
$$
m=2^a 3^b \prod_{i=1}^\ell p_i^{e_i} \prod_{j=1}^n q_j^{f_j} \prod_{k=1}^t r_k^{g_k}
$$
where the $p_i\equiv\pm7,\pm11\pmod{24}$,
      the $q_j\equiv1,19\pmod{24}$,
and 
      the $r_k\equiv5,23\pmod{24}$. 
Then 
$$
H_{\sqrt{6}}(m) = 
\begin{cases}
0, & \mbox{if some $e_i$ is odd or $a + \sum g_k$ is odd},\\
\prod_{j=1}^{n}(f_j+1) \prod_{k=1}^{t} (g_k + 1), & \mbox{otherwise.}
\end{cases}
$$
\end{lemma}

\begin{cor}
\label{corpell1}
Let $24n-1$ have the prime factorization
$$
24n-1=\prod_{i=1}^{k}p_i^{a_i}\prod_{j=1}^{l}q_j^{b_j}
$$
where $p_i\equiv 1,5,19,23\pmod{24}$ and $q_j\equiv 7,11,13,17\pmod{24}$. If 
some $b_j$ is odd then $d_0(n)$=0, otherwise we have
\beq
\label{6y2}
d_0(n)=\frac{1}{2}\prod_{i=1}^{k}(a_i+1).
\eeq
\end{cor}

\begin{proof}
We note $u^2 - 6v^2 \equiv -2 \pmod{48}$ if and only if $u\equiv\pm2\pmod{12}$
and $v$ is odd. Thus by \myeqref{eq:H6gen} we find that
\begin{align}
\label{eq:H6d0id}
\sum_{n=1}^\infty H_{\sqrt{6}}(48n-2) 
&=\sum_{n=1}^\infty \left(\sum_{-2n<m\leq 2n}q^{(12n-2)^2-6(2m-1)^2}
     +\sum_{-2n<m\leq 2n}q^{(12n+2)^2-6(2m-1)^2}\right)
\\
&=2\sum_{n=1}^\infty\left(
    \sum_{-n<m\leq n}q^{(12n-2)^2-6(4m-1)^2}
   +\sum_{-n<m\leq n}q^{(12n+2)^2-6(4m-1)^2}\right)
\nonumber\\
&=2 \sum_{n=1}^{\infty}d_0(n)q^{48n-2}.
\nonumber
\end{align}

If some $b_j$ is odd then by Lemma \myref{lem:lovejoy6} 
$H_{\sqrt{6}}(48n-2)=0$
and $d_0(n)=0$. We assume all the $b_j$ are even and note that  
$p^2\equiv1\pmod{24}$ for $(p,6)=1$. Next we show that the number of primes congruent to
$5$ or $23$ (counted with multiplicity) in the factorization of $48n-2$
is odd. If it is even then the product of these primes is either $1$
or $5\cdot23\equiv19\pmod{24}$. But this would imply that $24n-1$
is either $1$ or $19\pmod{24}$, which is a contradiction. The result
follows from \myeqref{eq:H6d0id} and Lemma \myref{lem:lovejoy6}.
\end{proof}

We also need Andrews, Dyson and Hickerson's \cite{An-Dy-Hi88} results for
$$
\sum_{n=1}^\infty S^{*}(n) q^n = \sum_{n=1}^\infty \frac{(-1)^nq^{n^2}}
                                                        {(q;q^2)_n}.
$$

\begin{lemma}
\label{lempell2}
For $n \ge 1$ let $s(n):=d_1(n)-d_2(n)$ and
$$
24n-1=\prod_{i=1}^{k}p_i^{a_i}\prod_{j=1}^{l}q_j^{b_j}
$$
be the prime factorization of $24n-1$, where the $p_i\equiv \pm1\pmod{24}$, 
and the $q_j\equiv \pm5,\pm7,\pm11\pmod{24}$.
If some $b_j\equiv 1\pmod2$ then $s(n)=0$, otherwise we have
$$
\abs{s(n)}=\frac{1}{2}\prod_{i=1}^{k}(a_i+1).
$$
\end{lemma}

\begin{proof}
By \cite[Eq.(5.2)]{An-Dy-Hi88} we find that
\begin{align*}
&\sum_{n=1}^\infty S^{*}(n) q^n = 
\sum_{n=1}^\infty \sum_{j=0}^{2n-1} (-1)^n q^{n(3n-1)-j(j+1)/2}(1+q^n)\\
&=\sum_{n=1}^\infty \sum_{m=-n+1}^{n} (-1)^n q^{n(3n-1)-m(2m-1)}(1+q^n)
=\sum_{n=-\infty}^\infty \sum_{m=-\abs{n}+1}^{\abs{n}} (-1)^n q^{n(3n-1)-m(2m-1)}.
\end{align*}
After replacing $q$ by $-q$ we have
$$
\sum_{n=1}^\infty (-1)^{n+1} S^{*}(n) q^n =
\sum_{n=-\infty}^\infty \sum_{m=-\abs{n}+1}^{\abs{n}} (-1)^{m+n+1}  
q^{n(3n-1)-m(2m-1)} = \sum_{n=1}^\infty s(n) q^n.
$$            
Hence
$$
\abs{s(n)} = \abs{S^{*}(n)}.
$$
By \cite[Theorem 5]{An-Dy-Hi88} we have
$$
S^{*}(n) = \tfrac{1}{2} T(1-24n),
$$
where $T(m)$ is the excess of the number of solutions of \myeqref{eq:Pell}
satisfying $u+3v\equiv\pm1\pmod{12}$ over the number satisfying
$u+3v\equiv\pm5\pmod{12}$.                                     
The result then follows from \cite[Theorem 3]{An-Dy-Hi88} which 
gives a formula for $T(m)$ in terms of the prime factorisation of $m$ when
$m\equiv 1\pmod{6}$.
\end{proof}

We can determine the difference $D_1(q)$ by Lemma \myref{corpell1} and 
Lemma \myref{lempell2}. First we consider the
case when is square-free. 

\begin{lemma}
\label{lempsisqf}
Let $n$ be an integer such that $N_\psi(n)$ is even and $24n-1$ is square-free. Then $d_1(n)$ is odd if and only if
$$
24n-1=p_1p_2,
$$
where $p_1$ and $p_2$ are primes 
for which $\{p_1,p_2\}\equiv \{5,19\}\pmod{24}$.
\end{lemma}

\begin{proof}
Suppose $\Np(n)$ is even where $24n-1$ is a square-free positive integer.
Let
$$
24n-1=\prod_{i=1}^{k}p_i
$$
be the prime factorization of $24n-1$. 
If $k=1$ then Theorem \myref{thepsi1} implies that $N_\psi(n)$ is odd which is
a contradiction. Hence either $k=2$ or $k>2$. If $k>2$ then Lemmas
\myref{corpell1} and \myref{lempell2} imply that $d_0(n)\equiv s(n)\equiv 0\pmod4$ 
and    
$$
2d_1(n)=d_0(n)+s(n)\equiv 0\pmod4
\qquad\mbox{and}\qquad 
d_1\equiv 0 \pmod{2}.
$$

Now suppose $k=2$, so that $24n-1$ is a product of two primes $p_1$ and $p_2$.
We note that $p_1\equiv -p_2\pmod{24}$. There are three cases.
\subsection*{Case 1: $\{p_1,p_2\}\equiv \{1,23\}\pmod{24}$.} Lemmas
\myref{corpell1} and \myref{lempell2} imply that
$$
2d_1(n)=d_0(n)+s(n)=2 \pm 2 \equiv 0 \pmod{4},\qquad\mbox{and}\qquad 
d_1\equiv 0 \pmod{2}.
$$

\subsection*{Case 2: $\{p_1,p_2\}\equiv \{5,19\}\pmod{24}$.} Lemmas
\myref{corpell1} and \myref{lempell2} imply that
$$
2d_1(n)=d_0(n)+s(n)=2 + 0 = 2,\qquad\mbox{and}\qquad d_1=1.
$$

\subsection*{Case 3: $\{p_1,p_2\}\equiv \{7,17\}\pmod{24}$.} Lemmas
\myref{corpell1} and \myref{lempell2} imply that
$$
2d_1(n)=d_0(n)+s(n)=0 + 0 = 0,\qquad\mbox{and}\qquad d_1=0.
$$

\subsection*{Case 4: $\{p_1,p_2\}\equiv \{11,13\}\pmod{24}$.} Lemmas
\myref{corpell1} and \myref{lempell2} imply that
$$
2d_1(n)=d_0(n)+s(n)=0 + 0 = 0,\qquad\mbox{and}\qquad d_1=0.
$$

Hence we see that $d_1(n)$ is odd if and only if 
$\{p_1,p_2\}\equiv \{5,19\}\pmod{24}$.
\end{proof}

To extend Lemma \myref{lempsisqf} to the case when $24n-1$ is not square-free
we need an analog of Lemma \myref{theHid}. Fortunately we have
the following result.
\begin{theorem}[{\cite[Theorem 1.3(i)]{Ga13}}]
\label{thm:heckespt}
If $\ell\ge 5$ is prime then 
\beq
\spt(\ell^2 n - s_\ell)
+ \leg{3}{\ell}\, \leg{1-24n}{\ell}\, \spt(n)
+ \ell \,\spt\left( \frac{n + s_\ell}{\ell^2} \right)
\equiv
 \leg{3}{\ell} \, (1 + \ell) \, \spt(n) \pmod{72},
\mylabel{eq:sptmod72}
\eeq
where $s_{\ell} = \tfrac{1}{24}(\ell^2-1)$.
\end{theorem}

By \myeqref{eq:sptPsimod4} we have the following corollary.
\begin{cor}
\label{cor:heckepsimod4}
If $\ell\ge 5$ is prime then 
\begin{align}
&\Np(\ell^2 n - s_\ell)
+ (-1)^{s_\ell}\,\leg{3}{\ell}\, \leg{1-24n}{\ell}\, \Np(n)
+ \ell \,\Np\left( \frac{n + s_\ell}{\ell^2} \right)
\mylabel{eq:Npsimod4}
\\
& \qquad \equiv
(-1)^{s_\ell}\, \leg{3}{\ell} \, (1 + \ell) \, \Np(n) \pmod{4}.
\nonumber
\end{align}
\end{cor}

Following \cite[p.353]{At1968} we use the standard practice of rewriting 
an arithmetic
function in terms of $24n-1$. We write
$$
\Nptwid(m) = 
\begin{cases}
\Np(n), & \mbox{if $m=24n-1$},\\
0,      & \mbox{if $m<23$ or $m\not\equiv -1 \pmod{24}$ or $m$ is 
                non-integral}.  
\end{cases}
$$

We rewrite Corollary \myref{cor:heckepsimod4} in terms $\Nptwid$ and
derive some congruence properties.
\begin{lemma}
\label{lem:Nptwid}
Let $\ell \ge 5$ be prime. Then
\begin{enumerate}
\item[(i)]
\beq               
\mylabel{eq:Nptwidmod4}
\Nptwid(\ell^2n)
+ (-1)^{s_\ell}\,\leg{-3n}{\ell}\, \Nptwid(n)
+ \ell \,\Nptwid\left( \frac{n}{\ell^2} \right)
\equiv
(-1)^{s_\ell}\, \leg{3}{\ell} \, (1 + \ell) \, \Nptwid(n) \pmod{4}.
\eeq                 
\item[(ii)]
If $(\ell,n)=1$ and $\Nptwid(n)$ is even then
   \begin{enumerate}
    \item[(a)] $\displaystyle \Nptwid(\ell^2 n)\equiv \Nptwid(n) \pmod{4}$,
    \item[(b)] $\displaystyle \Nptwid(\ell^3 n)\equiv 0\pmod{4}$.
   \end{enumerate}
\item[(iii)]
$$
\Nptwid(\ell^4n) \equiv \Nptwid(n) \pmod{4}.
$$
\end{enumerate}
\end{lemma}
\begin{proof}
(i). The congruence \myeqref{eq:Nptwidmod4} follows immediately from 
\myeqref{eq:Npsimod4}.

(ii). Suppose $(n,\ell)=1$ and $\Nptwid(n)$ is even. Since $\ell$ is odd we
have $(1+\ell)\Nptwid(n)\equiv0\pmod{4}$, and $\Nptwid(n/\ell^s)=0$ 
since
$(n,\ell)=1$. The congruence \myeqref{eq:Nptwidmod4} implies that
$$
\Nptwid(\ell^2 n)\equiv \Nptwid(n) \pmod{4}.
$$
Replacing $n$ by $\ell n$ in \myeqref{eq:Nptwidmod4} gives
$$
\Nptwid(\ell^3n)
\equiv
(-1)^{s_\ell}\, \leg{3}{\ell} \, (1 + \ell) \, \Nptwid(\ell n) \pmod{4}.
$$
Either $\Nptwid(\ell n)$ is even or $\Nptwid(\ell n)$ is odd and 
$\ell\equiv -1 \pmod{24}$. In both cases we have 
$(1 + \ell) \, \Nptwid(\ell n) \equiv 0\pmod{4}$
and 
$$
\Nptwid(\ell^3n)\equiv 0 \pmod{4}.
$$

(iii).  Replacing $n$ by $\ell^2 n$ in \myeqref{eq:Npsimod4} we have
\beq           
\Nptwid(\ell^4n)
\equiv
-\ell \Nptwid(n) +
(-1)^{s_\ell}\, \leg{3}{\ell} \, (1 + \ell) \, \Nptwid(\ell^2 n) \pmod{4}.
\label{eq:Npsimod4b}
\eeq
There are three cases.

\textbf{Case 1:} {$\ell\equiv 3\pmod{4}$.} 
From \myeqref{eq:Npsimod4b} we
have   
$$             
\Nptwid(\ell^4n) \equiv  \Nptwid(n) \pmod{4},
$$  
since $-\ell\equiv 1\pmod{4}$ and $(1+\ell)\equiv 0 \pmod{4}$.

\textbf{Case 2:} {$\ell\equiv 1\pmod{4}$ and $\Nptwid(\ell^2 n)$ is odd.} 
Since $\ell\not\equiv23\pmod{24}$ Theorem \myref{thepsi1} implies that
$\Nptwid(n)$ is also odd, and from \myref{eq:Npsimod4b} we have 
\begin{align*}
\Nptwid(\ell^4n)
&\equiv
\Nptwid(n) - (1+\ell)\left(\Nptwid(n) - 
(-1)^{s_\ell}\, \leg{3}{\ell} \Nptwid(\ell^2 n) \right) \pmod{4}\\
&\equiv 
\Nptwid(n) \pmod{4}.                              
\end{align*}

\textbf{Case 3:} {$\ell\equiv 1\pmod{4}$ and $\Nptwid(\ell^2 n)$ is even.} 
Again since $\ell\not\equiv23\pmod{24}$ Theorem \myref{thepsi1} implies that
$\Nptwid(n)$ is also even and from \myref{eq:Npsimod4b} we see that
$$
\Nptwid(\ell^4n) \equiv \Nptwid(n) \pmod{4}.                              
$$

This completes the proof of (iii) in all cases.
\end{proof}

\begin{theorem}
\label{thepsi2}
For $n>0$ be an integer, $N_{\psi}(n)\equiv 2\pmod4$ if and only if $24n-1$ has the form
$$
24n-1=p_1^{4a+1}p_2^{4b+1}m^2,
$$
where $p_1$ and $p_2$ are primes such that 
$\displaystyle\leg{p_1}{p_2}=-\varepsilon(p_2)$ for $\varepsilon(p)=-1$ 
if $p\equiv \pm5\pmod{24}$ and $\varepsilon(p)=1$ otherwise, 
$(m,p_1p_2)=1$ and $a,b\geq 0$ are integers.
\end{theorem}

\begin{proof}
    Suppose $n>0$. First we prove the result when $24n-1$ is square-free.
By Lemma \myref{lemD1} 
\beq
\Npm(n) - \Np(n) \equiv 2d_1(n) \pmod{4}.
\label{eq:Npmpmod4}
\eeq
We assume $24n-1$ is square-free and $\Np(n)\equiv 2\pmod{4}$. 
The congruence \myeqref{eq:Npmpmod4} implies that $\Npm(n)\equiv0$ or $2\pmod{4}$.
\vskip 0pt\noindent\textbf{Case 1:} {$\Npm(n)\equiv0\pmod{4}$.} Therefore
$d_1(n)$ is odd and Lemma \myref{lempsisqf} implies that $24n-1=p_1 p_2$ for 
primes $p_1$, $p_2$ satisfying $\{p_1,p_2\}\equiv \{5,19\}\pmod{24}$. Theorem
\myref{the12phi-1} implies that $\leg{p_1}{p_2}=1=-\varepsilon(p_2)$.
\vskip 0pt\noindent\textbf{Case 2:} {$\Npm(n)\equiv2\pmod{4}$.} By Theorem
\myref{the12phi-1} there are primes $p_1$, $p_2$ satisfying
$\leg{p_1}{p_2}=-1$ and $24n-1=p_1p_2$. By \myeqref{eq:Npmpmod4} $d_1(n)$ is
even and Lemma \myref{lempsisqf} implies that $p_1$, 
$p_2\not\equiv\pm5\pmod{24}$, so that $\leg{p_1}{p_2}
=-\varepsilon(p_2)$. 

Similarly we can show the converse that if $24n-1=p_1p_2$
where $p_1$ and $p_2$ are primes satisfying  $\leg{p_1}{p_2}
=-\varepsilon(p_2)$, then $\Np(n)\equiv 2\pmod{4}$. This completes
the proof of the result when $24n-1$ is square-free.

We now consider the general case. For this part we first assume that
$\Np(n)\equiv 2\pmod{4}$. We let $M=24n-1$ so that $\Nptwid(M)=\Np(n)$ and
$\Nptwid(M)\equiv 2\pmod{4}$. We write the prime factorization of $M$ as
$$
M  = \prod_{i=1}^k p_i^{4a_i+r_i},
$$
where the $a_i$, $r_i$ are integers satisfying $0 \le r_i \le 3$.
Then by Lemma \myref{lem:Nptwid}(iii) we have
$$
\Nptwid(M) \equiv \Nptwid(M') \equiv 2 \pmod{4},
$$
where
$$
M'  = \prod_{i=1}^k p_i^{r_i} \equiv -1\pmod{24}.
$$
Lemma \myref{lem:Nptwid}(ii)(b) implies that none of the $r_i$ are equal
to $3$.  
Lemma \myref{lem:Nptwid}(ii)(a) implies that                                 
$$
\Nptwid(M) \equiv \Nptwid(M'') \equiv 2 \pmod{4},
$$
where
$$
M''  = \prod_{j \in J} p_i \equiv -1\pmod{24},
$$
and $J$ is the set of $j$ for which $r_j=1$. Here $M''$ is square-free.
Therefore  $J$ is a set of two primes. Without loss of generality 
we may assume these two primes are $p_1$ and $p_2$ where $\leg{p_1}{p_2}
=-\varepsilon(p_2)$. Hence
$$
24n - 1 = M = p_1^{4a_1+1} p_2^{4a_2+1} m^2,
$$
where $(m,p_1p_2)=1$ and $24n-1$ has the desired form.

Finally we assume 
$$
M=24n-1=p_1^{4a+1}p_2^{4b+1}m^2,
$$
where $p_1$ and $p_2$ are primes such that 
$\displaystyle\leg{p_1}{p_2}=-\varepsilon(p_2)$.
Arguing as before and using Lemma \myref{lem:Nptwid}(ii),(iii) we have
$$
\Np(n) = \Nptwid(M) \equiv \Nptwid(p_1 p_2) \pmod{4},
$$
where $M \equiv p_1 p_2 \equiv -1 \pmod{24}$.
Since $p_1 p_2$ is square-free, satisfies $\leg{p_1}{p_2}=-\varepsilon(p_2)$
and we have proved the square-free case $\Nptwid(p_1 p_2)\equiv 2 \pmod{4}$,
the result
$$
\Np(n) \equiv 2 \pmod{4}
$$
follows.
\end{proof}

We are now ready to complete the proof of Conjecture \myref{sptconj}
and related congruences for $\psi(q)$. 
\begin{theorem}
\label{thepsic}
Let $p>3$ be a prime where $p\not\equiv23\pmod{24}$. Suppose 
$24\delta_p\equiv1\pmod{p^2}$, $k,n\in \mathbb{Z}$ and 
$\displaystyle\leg{k}{p}=\varepsilon(p)$ where $\varepsilon(p)=-1$ if 
$p\equiv \pm5\pmod{24}$ and $\varepsilon(p)=1$ otherwise. Then
\begin{align}
\label{Npid}
N_{\psi}(p^2n+(pk+1)\delta_p)&\equiv 0\pmod4,\\
\label{sptid}
\mathrm{spt}(p^2n+(pk+1)\delta_p)&\equiv 0\pmod4.
\end{align}
\end{theorem}
\begin{remark}
We have rewritten the statement of Conjecture \myref{sptconj} in 
an equivalent form. The equivalence follows from the observations that
$\leg{6}{p}=1$ for $p\equiv 1,\pm5\pmod{24}$ and 
$\leg{6}{p}=-1$ for $p\equiv\pm7,\pm11\pmod{24}$. We note that a weak version
of the mod $4$  congruences for $\Np(n)$ were conjectured by Bryson, Ono 
Pitman and Rhoades \cite[Eq.(1.7), Conjecture 1.6]{Br-On-Pi-Rh12}.
We discuss this further in Section \myref{subsec:stronguni}.
\end{remark}

\begin{proof}
Assume $p>3$ is prime, $p\not\equiv23\pmod{24}$, $24\delta_p\equiv1\pmod{p^2}$, 
$k,n\in \mathbb{Z}$ and 
$\displaystyle\leg{k}{p}=\varepsilon(p)$. 
We let $m=p^2n+(pk+1)\delta_p$, so that
\beq
\label{p24m-1}
24m-1\equiv pk\pmod{p^2}.
\eeq
By Theorem \myref{thepsi1} we see that $N_{\psi}(m)$ is even since
$24m-1$ and  $p\not\equiv23\pmod{24}$. Now suppose that 
$\Np(m)\equiv 2\pmod{4}$. By Theorem \myref{thepsi2} 
$N_{\psi}(m)\equiv 2\pmod4$ if and only if $24m-1$ has the form
$$
24m-1=p^1 q^{4b+1}t^2,
$$
where $q$ is a prime satisfying $\displaystyle\leg{p}{q}=-\varepsilon(q)$ 
$(pq,t)=1$, and $b$, $t>0$ are integers. Either $p$ or $q\equiv1\pmod{4}$
so by quadratic reciprocity $\leg{q}{p}=\leg{p}{q}=-\varepsilon(q)$.
From \myeqref{p24m-1} we have
$$
k \equiv q^{4b+1}t^2\pmod{p}.
$$
But this implies that
$$
\leg{q}{p} = \leg{k}{p} = \varepsilon(q),
$$
which is a contradiction. 
Hence $\Np(m)\not\equiv 2\pmod4$ and
$$
N_{\psi}(m)\equiv 0\pmod4,
$$
which is \myeqref{Npid}. Finally the result \myeqref{sptid} holds by 
\myeqref{eq:sptPsimod4}.
\end{proof}

\subsection{Strongly unimodal sequences}
\label{subsec:stronguni}

In this section we prove  Bryson, Ono, Pitman and Rhoades's Conjecture
\myref{conju}. This conjecture has three parts. In this section 
we will prove the first part, congruence \myeqref{eq:conjuA}. The third
part, congruence \myeqref{eq:conjuC}, follows from Theorem \myref{thepsic}.
The second part, congruence \myeqref{eq:conjuB}, will follow from the
first and third parts. As noted in Section \myref{sec:intro}
we have
\beq
\mathcal{U}(\pm i;q) = \psi(q) = 
\sum_{n=1}^\infty N_\psi(n) q^n =
\sum_{n=1}^\infty \frac{q^{n^2}}{(q;q^2)_n}.
\label{eq:UpsiidB}
\eeq
See \cite[p.16064]{Br-On-Pi-Rh12}. Also define $U(q)$ by
\beq
U(q) := \mathcal{U}(1;q) = \sum_{n=1}^\infty u(n) q^n.
\label{eq:Uqdef}
\eeq
Letting $z=i$ in \myeqref{eq:Uzqdef} we have
\beq
u(0,4;n) - u(2,4,n) = \Np(n),\quad\mbox{and}\quad u(1,4;n)=u(3,4;n).
\label{eq:u024psi}
\eeq
Hence
\beq
u(n) = u(0,4;n) + 2u(1,4;n) + u(2,4;n),
\label{eq:un0124}
\eeq
and
\beq
u(n) + \Np(n) = 2(u(0,4;n)+u(1,4;n)).
\label{eq:unpsi}
\eeq
From \myeqref{eq:u024psi} it is clear that \myeqref{eq:conjuC} 
follows from Theorem \myref{thepsic}. Equations 
\myeqref{eq:u024psi}--\myeqref{eq:unpsi} together with Theorem \myref{thepsic}
show how \myeqref{eq:conjuB} follows from \myeqref{eq:conjuA} and
\myeqref{eq:conjuC}. Here we assume 
that $\ell\equiv 7,11,13,17\pmod{24}$ is prime and
$\displaystyle\leg{k}{\ell}=-1$. 

We point out that there is a stronger result for Bryson, Ono, Pitman's
Conjecture \myeqref{eq:conjuC} that includes primes 
$\ell\equiv 1,\pm5\pmod{24}$.
This follows easily from \myeqref{Npid} and \myeqref{eq:u024psi}.
We state this as a separate theorem.
\begin{theorem}
\label{conjuCstrong}
Suppose $\ell>$ is prime and $\ell\not\equiv 23\pmod{24}$. Then
for $n\ge0$ we have
\beq                
u(0,4; \ell^2n+kl-s(\ell))\equiv 
u(2,4; \ell^2n+kl-s(\ell))\pmod4,
\label{eq:conjuCb}
\eeq          
provided $\displaystyle\leg{k}{\ell}=\eptwid(\ell)$ and 
$s(\ell)=\tfrac{1}{24}(\ell^2-1)$.
\end{theorem}
\begin{remark}
Here $\eptwid(\ell)$ is defined in the statement of 
Conjecture \myref{sptconj}; i.e.\ 
$\eptwid(\ell) =  1$ if $\ell\equiv 1\pmod{24}$ and $-1$ otherwise. 
For example let $\ell=457\equiv 1 \pmod{24}$, and $k=21$. 
Then $\leg{k}{\ell}=\leg{21}{457}=1 =\eptwid(\ell)$, $s(\ell)=8702$, and
$k\ell -s(\ell)=895$. We have
\begin{align*}
u(0,4;895) &= 256203223294825619203431487908 \equiv 0 \pmod{4}\\
u(1,4;895)=u(3,4,895)&= 256203223294825426775345978961 \equiv 1\pmod{4}\\
u(2,4;895)&=256203223294825234347260470016 \equiv 0 \pmod{4}\\
u(895) &= 1024812893179301707101383915846\equiv 2\pmod{4}.
\end{align*}
\end{remark}

Hikami and Lovejoy \cite[Theorem 4.1]{Hi-Lo2015} found a Hecke-Rogers
identity for the generating function $\mathcal{U}(z;q)$;
\beq
\label{eq:Uzq}
(1+z)\mathcal{U}(z;q)=\frac{q}{(q;q)_\infty}
\left(\sum_{r,n\geq0}-\sum_{r,n<0}\right)(-1)^nz^{-r}q^{n(3n+5)/2+2nr+r(r+3)/2}.
\eeq
See also \cite[Eq.(1.3)]{Ki-Li-Lo16}.
Considering Lemma \myref{lem:psiHR} and that $\psi(q)=\mathcal{U}(i;q)$ it
is reasonable to suspect that $\mathcal{U}(1;q)$ has a similar
Hecke-Rogers identity. 

\begin{lemma}
\label{lem:altUHR}
We have
\beq
\label{eq:altuhec}
J_1\,U(q)
=\sum_{n=1}^{\infty}\sum_{r=1}^{n}(-1)^{r-1}q^{2n^2-n-r(r-1)/2}(1+q^{2n}).
\eeq
\end{lemma}
\begin{proof}
We define
$$
Q(n,r) := n(3n+5)/2 + 2nr + r(r+3)/2+1,
$$
and
$$
\mathcal{F}(z;q) = \sum_{r=0}^\infty\sum_{n=0}^\infty (-1)^n z^{-r} q^{Q(n,r)}.
$$
Thus \myeqref{eq:Uzq} can be rewritten as
\beq
(1+z)\mathcal{U}(z;q)=\frac{1}{(q;q)_\infty}
\left(\mathcal{F}(z;q) + z\mathcal{F}(z^{-1};q)\right),
\label{eq:UFzid}
\eeq
since
$$
Q(n,r) = Q(-n-1,-r-1).
$$
We observe that
$$
Q(n-1,r-2n+1) = r(r+1)/2 - n(n-1)/2,
$$
so that
$$
\mathcal{F}(z;q) = 
\sum_{n=1}^\infty \sum_{r=2n-1}^\infty (-1)^{n+1} z^{2n-r-1} 
q^{r(r+1)/2 - n(n-1)/2}.
$$
Letting $z=1$ in \myeqref{eq:UFzid} we find that
\begin{align*}
J_1 U(q) &= J_1 \mathcal{U}(1;q) = \mathcal{F}(1;q) 
= 
\sum_{n=1}^\infty \sum_{r=2n-1}^\infty (-1)^{n+1} 
q^{r(r+1)/2 - n(n-1)/2}\\
&=
\sum_{r=1}^\infty\sum_{n=1}^{\FL{(r+1)/2}} (-1)^{n+1} 
q^{r(r+1)/2 - n(n-1)/2}\\
&=
\sum_{r=1}^\infty\sum_{n=1}^{r} (-1)^{n+1} \left(
q^{r(2r+1) - n(n-1)/2}+ q^{r(2r-1) - n(n-1)/2}\right),
\end{align*}
by replacing $r$ by $2r$ and $2r-1$ in the previous sum.
The result \myeqref{eq:altuhec} follows easily by interchanging $r$ and $n$.
\end{proof}

The proof of the next lemma, which contains our new Hecke-Rogers identity
for $U(q)$, is analogous to that of Lemma \myref{lem:psiHR}.
\begin{lemma}
\label{lem:newUHR}
We have
\beq
\label{uhec}
\frac{J_1^2}{J_2}U(q)
=\sum_{n=1}^{\infty}\sum_{m=1-n}^{n}\sg(m)(-1)^{n-1}q^{n(3n-1)-2m^2+m}(1+q^{2n}),
\eeq
where $\sg(m)=1$ if $m>0$ and $\sg(m)=-1$ otherwise.
\end{lemma}

\begin{proof}
We define
$$
A_k:=
\sum_{n=1}^{\infty}\sum_{r=1}^{n}(-1)^{r+k-1}q^{2n^2-n-(r-k)(r-k-1)/2}(1+q^{2n}),
$$
so that $A_0 = J_1 U(q)$ by Lemma \myref{lem:altUHR}.
Using \myeqref{jacp} and \myeqref{eq:pentid}
we derive a recurrence relation for $A_k$:
\begin{align*}
&A_k+A_{-k}-(A_{k-1}+A_{1-k})\\
&=\sum_{n=1}^{\infty}\sum_{r=1}^{n}(-1)^{r+k-1}(q^{2n^2-n-(r-k)(r-k-1)/2}+q^{2n^2-n-(r+k)(r+k-1)/2})(1+q^{2n})\\
&-\sum_{n=1}^{\infty}\sum_{r=1}^{n}(-1)^{r+k}(q^{2n^2-n-(r-k+1)(r-k)/2}+q^{2n^2-n-(r+k-1)(r+k-2)/2})(1+q^{2n})\\
&=(-1)^k\sum_{n=1}^{\infty}(q^{2n^2-n-k(k-1)/2}+q^{2n^2-n-k(k-1)/2})(1+q^{2n})\\
&-(-1)^k\sum_{n=1}^{\infty}(-1)^n(q^{2n^2-n-(n-k)(n-k+1)/2}+q^{2n^2-n-(n+k)(n+k-1)/2})(1+q^{2n})\\
&=(-1)^k\left( 2\sum_{n=-\infty}^\infty q^{2n^2-n-k(k-1)/2} -
   \sum_{n=-\infty}^\infty (-1)^n q^{n(3n+1)/2 + n(1-k) -k(k-1)/2}\right.\\
&{\hskip 2in} \left. - \sum_{n=-\infty}^\infty (-1)^nq^{n(3n+1)/2 + nk -k(k-1)/2}
\right)\\
&=2\frac{J_2^2}{J_1}(-1)^kq^{-k(k-1)/2}-J_1a_k.
\end{align*}
where
\begin{align*}
&a_{3r}:=q^{-6r^2+r},\\
&a_{3r+1}:=-q^{-6r^2-r},\\
&a_{3r+2}:=q^{-6r^2-5r-1}-q^{-6r^2-7r-2}.
\end{align*}
Therefore for $k>0$ we have
\beq
\label{ak-k}
A_k+A_{-k}=2A_0+2\frac{J_2^2}{J_1}\sum_{r=1}^{k}(-1)^rq^{-r(r-1)/2}-J_1\sum_{r=1}^{k}a_r.
\eeq
We calculate the following sum two ways. By using \myeqref{jacp} we have
\begin{align}
\label{Ak01}
&\sum_{k=-\infty}^{\infty}(-1)^kA_kq^{k^2}=
\sum_{k=-\infty}^{\infty}\sum_{n=1}^{\infty}
\sum_{r=1}^{n}(-1)^{r-1}q^{2n^2-n+k^2-(r-k)(r-k-1)/2}(1+q^{2n})\\
&=\sum_{k=-\infty}^{\infty}\sum_{n=1}^{\infty}
  \sum_{r=1}^{n}(-1)^{r-1}
  q^{2n^2-n+(k+r)(k+r-1)/2-r(r-1)}(1+q^{2n})
\nonumber \\
&=\sum_{k=-\infty}^{\infty}q^{k(k-1)/2} \sum_{n=1}^{\infty}
  \sum_{r=1}^{n}(-1)^{r-1}
  q^{2n^2-n-r(r-1)}(1+q^{2n})
\nonumber \\
\nonumber
&=2\frac{J_2^2}{J_1}
\sum_{n=1}^{\infty}\sum_{r=1}^{n}(-1)^{r-1}q^{2n^2-n-r(r-1)}(1+q^{2n})\\
&=2\frac{J_2^2}{J_1}
\sum_{n=-\infty}^{\infty}\sum_{r=1}^{\abs{n}}(-1)^{r-1}q^{2n^2-n-r(r-1)}.
\nonumber
\end{align}
This time we let
$$
Q(n,r) = 2n^2 - n - r(r-1),
$$
and observe that
$$
Q(k-j, k - 2j) = k^2 - 2j^2 - j,\qquad
Q(j-k, k - 2j+1) = k^2 - 2j^2 + j.         
$$
It follows that
\begin{align*}
\sum_{n=-\infty}^{\infty}\sum_{r=1}^{\abs{n}}(-1)^{r-1}q^{2n^2-n-r(r-1)}
&=\sum_{n=1}^{\infty}\sum_{r=1}^{{n}}(-1)^{r-1}q^{2n^2-n-r(r-1)} +
  \sum_{n=-\infty}^{-1}\sum_{r=1}^{{-n}}(-1)^{r-1}q^{2n^2-n-r(r-1)}\\
&=\sum_{j=0}^\infty \sum_{k=2j+1}^\infty  (-1)^{k+1}q^{k^2-2j^2-j} 
 +\sum_{j=1}^\infty \sum_{k=2j}^\infty  (-1)^{k}q^{k^2-2j^2+j} \\
&=\sum_{m=1}^\infty \sum_{k=m}^\infty  (-1)^{m+k}q^{k^2-m(m-1)/2} 
=\sum_{k=1}^\infty \sum_{j=1}^k  (-1)^{j+k}q^{k^2-j(j-1)/2},    
\end{align*}
and from \myeqref{Ak01} we have
$$
\sum_{k=-\infty}^{\infty}(-1)^kA_kq^{k^2}=
2\frac{J_2^2}{J_1}
\sum_{k=1}^\infty \sum_{j=1}^k  (-1)^{j+k}q^{k^2-j(j-1)/2}.
$$
Next we calculate the sum on the left side of \myeqref{Ak01} using
\myeqref{ak-k} and find that
\beq          
\label{Ak02}
\sum_{k=-\infty}^{\infty}(-1)^kA_kq^{k^2}=\frac{J_1^2}{J_2}A_0-J_1
\sum_{k=1}^{\infty}\sum_{r=1}^{k}(-1)^kq^{k^2}a_r
+2\frac{J_2^2}{J_1}\sum_{k=1}^{\infty}
\sum_{r=1}^{k}(-1)^{r+k}q^{k^2-r(r-1)/2}.
\eeq
By \myeqref{Ak01} and \myeqref{Ak02} we have
\beq
\label{Uq1m}
\frac{J_1^2}{J_2}A_0=J_1\sum_{k=1}^{\infty}\sum_{r=1}^{k}(-1)^kq^{k^2}a_r.
\eeq
Next we proceed as in the proof of Lemma \myref{psihec} and define
analogous functions $G(r,k):=(-1)^kq^{k^2}a_r$ and 
$F(m,n):=\sg(m) (-1)^{n-1}q^{n(3n-1)-2m^2+m}$. The analogs of
equations \myeqref{gf1}--\myeqref{gf3} hold and we have
\begin{align}
&\frac{J_1^2}{J_2}U(q)
=\sum_{k=1}^{\infty}\sum_{r=1}^{k}(-1)^kq^{k^2}a_r
=\sum_{k=1}^\infty\sum_{r=1}^k G(r,k)
=\sum_{n=-\infty}^\infty\sum_{m=1-\abs{n}}^{\abs{n}}F(m,n)
	\label{eq:UGFHR}\\
\nonumber
&=\sum_{n=1}^{\infty}\sum_{m=1-n}^{n}\sg(m)(-1)^{n-1}q^{n(3n-1)-2m^2+m}(1+q^{2n}).
\end{align}
\end{proof}

We are now ready to complete the proof of Conjecture \myref{conju}.

\begin{proof}[Proof of congruence \myeqref{eq:conjuA} in 
	Conjecture \myref{conju}]

We need a result similar to Lemma \myref{lemD1}. We define
$$
	\epsilon_1(m,n) = \frac{1}{2}\Parans{\sg(m) (-1)^{n-1} -
	\sg(n) (-1)^{m-1}}.
$$
We note that $\epsilon_1(m,n)$ only takes the values $0$, $\pm1$.
Then by \myeqref{uhec} and \myeqref{psihec} we have
\beq
U(q) - \psi(q) \equiv 2\, D_u(q) \pmod{4},            
\label{eq:Upsimod4}
\eeq
where
$$
D_u(q) :=\sum_{n=1}^\infty d_u(n) q^n := 
\sum_{n=-\infty}^\infty 
\sum_{1-\abs{n}\leq m\leq\abs{n}}\epsilon_1(m,n)q^{n(3n-1)-m(2m-1)}.
$$
The proof of \myeqref{eq:Upsimod4} is completely analogous to that
of Lemma \myref{lemD1}.

We assume $\ell\equiv 7,11,13,17\pmod{24}$ is prime and
let $n$, $k$ be  integers where $\displaystyle\leg{k}{\ell}=-1$. 
We suppose that
$$
m = \ell^2n+kl-s(\ell),
$$
and recall that $s(\ell)=\tfrac{1}{24}(\ell^2-1)$.
Then
$$
24m-1 = \ell ( \ell(24n-1) + 24 k)\equiv 24 k \ell \pmod{\ell^2}.
$$
Hence $\ell\|24m-1$ and $d_0(m)=0$ by Corollary \myref{corpell1},
since $\ell\equiv\pm7,\pm11\pmod{24}$. We note that
$$
k\ell - s(\ell) \equiv (24 k\ell+1) \delta_\ell \pmod{\ell^2}.
$$
Therefore
$$
\Np(m) = \Np(\ell^2 n + k\ell -s(\ell)) \equiv 0 \pmod{4}
$$
by Theorem \myref{thepsic} since $\ell\equiv\pm7,\pm11\pmod{24}$ and
$$
\leg{24k}{\ell} = \leg{6}{\ell} \leg{k}{\ell} = 1 = \varepsilon(\ell).
$$
Since $d_0(m)=0$ we have $d_u(m)=0$ and by
\myeqref{eq:Upsimod4} we have
$$
u(m) \equiv \Np(m) + 2 d_u(m) \equiv \Np(m) \pmod{4},
$$
and
$$
u( \ell^2n+kl-s(\ell)) = u(m) \equiv 0 \pmod{4}.
$$
\end{proof}

\section{Conclusion}
\label{sec:end}
The main goal of this paper was to prove the mod $4$ unimodal sequence 
conjectures of Bryson, Ono, Pitman and Rhoades \cite{Br-On-Pi-Rh12}
and Kim, Lim and Lovejoy \cite{Ki-Li-Lo16}. We also proved
a related mod $4$ conjecture for the Andrews spt-function. The crucial
part of the proofs was the connection with the Hurwitz class number.
Along the way we needed to study the mod $4$ behaviour of the
coefficients of certain mock theta functions. As mentioned before, the
parity of these  was determined very recently by Wang \cite{Wa2020}.
It would be interesting to determine whether the methods of this paper
can be used to extend Wang's parity results to mod $4$ results for
other mock theta functions.

In this paper we have found a number new Hecke-Rogers identities
\myeqref{psihec}, \myeqref{uhec} and \myeqn{altuhec}. Can these identities
be proved using the Bailey pair machinery \cite[Ch.3]{An1986}?
What are the missing Bailey pairs?

\subsection*{Acknowledgments}
We would like to thank Jonathan Bradley-Thrush,
Chris Jennings-Shaffer, Jeremy Lovejoy  and Eric Mortenson
for their comments and suggestions. 



\end{document}